\documentclass[12pt,reqno]{amsart}

\usepackage{color}
\definecolor{vert}{rgb}{0,0.6,0}

\usepackage{comment}
\usepackage{amsmath}
\usepackage{amssymb}
\usepackage{amsthm}
\usepackage[latin1]{inputenc}
\usepackage{eurosym}
\usepackage{graphicx}
\usepackage{bmpsize}
\usepackage{epsfig}
\usepackage{hyperref}
\usepackage{dsfont}
\usepackage[nocompress, space]{cite}
\usepackage{caption}
\usepackage{subcaption}
\usepackage{cases}

\usepackage[displaymath,mathlines]{lineno}

\allowdisplaybreaks

\usepackage{ifthen}

\newcommand{\R}{\mathbb{R}}

\newcommand{\cA}{\mathcal{A}}

\newcommand{\ep}{\varepsilon}

\newcommand{\AC}{{\rm AC\,}}

\newcommand{\ol}{\overline}
\newcommand{\ul}{\underline}

\newcommand{\argmax}{\operatorname{argmax}}

\newcommand{\supp}{\operatorname{supp}}

\renewcommand{\div}{\operatorname{div}}

\def\leq{\leqslant}
\def\geq{\geqslant}

\numberwithin{equation}{section}

\newtheoremstyle{thmlemcorr}{10pt}{10pt}{\itshape}{}{\bfseries}{.}{10pt}{{\thmname{#1}\thmnumber{
#2}\thmnote{ (#3)}}}
\newtheoremstyle{thmlemcorr*}{10pt}{10pt}{\itshape}{}{\bfseries}{.}\newline{{\thmname{#1}\thmnumber{
#2}\thmnote{ (#3)}}}
\newtheoremstyle{defi}{10pt}{10pt}{\itshape}{}{\bfseries}{.}{10pt}{{\thmname{#1}\thmnumber{
#2}\thmnote{ (#3)}}}
\newtheoremstyle{remexample}{10pt}{10pt}{}{}{\bfseries}{.}{10pt}{{\thmname{#1}\thmnumber{
#2}\thmnote{ (#3)}}}
\newtheoremstyle{ass}{10pt}{10pt}{}{}{\bfseries}{.}{10pt}{{\thmname{#1}\thmnumber{
A#2}\thmnote{ (#3)}}}

\theoremstyle{thmlemcorr}
\newtheorem{theorem}{Theorem}
\numberwithin{theorem}{section}
\newtheorem{lemma}[theorem]{Lemma}
\newtheorem{corollary}[theorem]{Corollary}
\newtheorem{proposition}[theorem]{Proposition}

\theoremstyle{thmlemcorr*}
\newtheorem{theorem*}{Theorem}
\newtheorem{lemma*}[theorem]{Lemma}
\newtheorem{corollary*}[theorem]{Corollary}
\newtheorem{proposition*}[theorem]{Proposition}
\newtheorem{problem*}[theorem]{Problem}
\newtheorem{conjecture*}[theorem]{Conjecture}

\theoremstyle{defi}
\newtheorem{definition}[theorem]{Definition}

\theoremstyle{remexample}
\newtheorem{remark}{Remark}
\newtheorem{example}[theorem]{Example}

\theoremstyle{plain}

\theoremstyle{ass}

\begin{document}

\title[Forced mean curvature flow]{Level-set forced mean curvature flow with the Neumann boundary condition}

\author{Jiwoong Jang}
\address[Jiwoong Jang]
{
	Department of Mathematics, 
	University of Wisconsin Madison, Van Vleck hall, 480 Lincoln drive, Madison, WI 53706, USA}
\email{jjang57@wisc.edu}

\author{Dohyun Kwon}
\address[Dohyun Kwon]
{
	Department of Mathematics, 
	University of Wisconsin Madison, Van Vleck hall, 480 Lincoln drive, Madison, WI 53706, USA}
\email{dkwon7@wisc.edu}

\author{Hiroyoshi Mitake}
\address[H. Mitake]{
	Graduate School of Mathematical Sciences, 
	University of Tokyo 
	3-8-1 Komaba, Meguro-ku, Tokyo, 153-8914, Japan}
\email{mitake@g.ecc.u-tokyo.ac.jp}

\author{Hung V. Tran}
\address[Hung V. Tran]
{
	Department of Mathematics, 
	University of Wisconsin Madison, Van Vleck hall, 480 Lincoln drive, Madison, WI 53706, USA}
\email{hung@math.wisc.edu}

\thanks{
	The work of JJ was partially supported by NSF CAREER grant DMS-1843320.
	The work of HM was partially supported by the JSPS grants: KAKENHI \#19K03580, \#19H00639, \#17KK0093, \#20H01816. 
	The work of HT was partially supported by  NSF CAREER grant DMS-1843320 and a Simons Fellowship.
}

\date{\today}

\begin{abstract}
Here, we study a level-set forced mean curvature flow with the homogeneous Neumann boundary condition.
We first show that the solution is Lipschitz in time and locally Lipschitz in space.
Then, under an additional condition on the forcing term, we prove that the solution is globally Lipschitz.
We obtain the large time behavior of the solution in this setting and study the large time profile in some specific situations.
Finally, we give two examples demonstrating that the additional condition on the forcing term is sharp, 
and without it, the solution might not be globally Lipschitz.
\end{abstract}

\maketitle


\section{Introduction} \label{sec:intro}
 In this paper, we study the level-set equation for the forced mean curvature flow
\begin{numcases}{}
u_t=|D u|\div \left(\frac{D u}{|D u|}\right)+c(x)|D u|  & \text{in}  $\Omega\times(0,\infty)$, \label{eq:levelset1}\\
\displaystyle \frac{\partial u}{\partial \Vec{\mathbf{n}}}=0 & \mbox{on}  $\partial\Omega\times[0,\infty)$, \label{eq:levelset2}\\
u(x,0)=u_0(x) & \mbox{on} $ \overline{\Omega}$.
\label{eq:levelset3} 
\end{numcases}
The domain $\Omega \subset \R^n$ with $n\ge2$ is assumed to be bounded and $C^{2,\theta}$ for some $\theta \in (0,1)$.
Here, $c=c(x)$ is a  forcing function, which is in $C^{1}(\overline{\Omega})$, and $\Vec{\mathbf{n}}$ is the outward unit normal vector to $\partial\Omega$. 
Throughout this paper, we assume that $u_0\in C^{2,\theta}(\overline{\Omega})$, and $\frac{\partial u_0}{\partial \Vec{\mathbf{n}}}=0\ \text{on}\ \partial\Omega$ for compatibility. 

We first notice that the well-posedness and the comparison principle for \eqref{eq:levelset1}--\eqref{eq:levelset3} are well established in the theory of viscosity solutions (see \cite{ES, CGG, GS, GS2} for instance). Our main interest in this paper is to go beyond the well-posedness theory to understand the Lipschitz regularity and large time behavior of the solution. The Lipschitz regularity for the solution is rather subtle because of the competition between the forcing term and the mean curvature term together with the constraint on perpendicular intersections of the level sets of the solution with the boundary of $\Omega$. It is worth emphasizing that the geometry of $\partial \Omega$ plays a crucial role in the analysis.

\medskip
We now describe our main results. First of all, we show that $u$ is Lipschitz in time and locally Lipschitz in space. 
\begin{theorem}\label{thm:local-grad}
Let $u$ be the unique viscosity solution $u$ of \eqref{eq:levelset1}--\eqref{eq:levelset3}.
Then, there exists a constant $M>0$ and for each $T>0$, there exists a constant $C_T>0$ depending on $T$ such that 
\[
\begin{cases}
|u(x,t)-u(x,s)|\leq M|t-s|,\\
|u(x,t)-u(y,t)|\leq C_T|x-y|,
\end{cases}
\quad \text{ for all $x,y\in\overline{\Omega}$, $t,s\in [0,T]$.}
\]
\end{theorem}

We next show that if we put some further conditions on the forcing term $c$, then we have the global Lipschitz estimate in $x$ of the solution.
Denote by
\begin{equation*}
\begin{cases}
C_0:=\max\{-\lambda\,:\,\lambda\ \textrm{is a principal curvature of $\partial \Omega$ at $x_0$ for $x_0 \in \partial  \Omega$}\}\in\mathbb{R},\\
K_0:=\min\{d\,:\, d\  \textrm{is the diameter of an open ball inscribed in}\ \Omega\}>0.
\end{cases}
\end{equation*}

\begin{theorem}\label{thm:global-grad}
Assume that  there exists $\delta>0$ such that
\begin{equation}\label{condition:c}
\frac{1}{n}c(x)^2-|Dc(x)|-\delta>\max\left\{0,\ C_0|c(x)|+\frac{2nC_0}{K_0}\right\}\quad\text{for all}\ x\in\Omega.
\end{equation}
Let $u$ be the unique viscosity solution to \eqref{eq:levelset1}--\eqref{eq:levelset3}.
 Then, there exist constants $M,L>0$ depending only on the forcing term $c$ and the initial data $u_0$ such that
\begin{equation}
\label{eq:lip}
\begin{cases}
|u(x,t)-u(x,s)|\leq M|t-s|,\\
|u(x,t)-u(y,t)|\leq L |x-y|,
\end{cases}
\quad \text{ for all $x,y\in\overline{\Omega}$, $t,s\in[0,\infty)$.} 
\end{equation}
\end{theorem}

Let us now explain a bit the geometric meaning of $K_0$.
For each $x\in \partial \Omega$, let 
\[
K_x = \max\{ 2r>0\,:\, B(x-r\Vec{\mathbf{n}}(x), r) \subset \Omega\}.
\] 
Then, $K_0 = \min_{x \in \partial \Omega} K_x$.
We notice next that  if $\Omega$ is convex in Theorem \ref{thm:global-grad}, then we clearly have $C_0\le0$. 
In this case, \eqref{condition:c} becomes $\frac{1}{n} c(x)^2 - |Dc(x)| -\delta>0$, a kind of coercive assumption, which often appears in the usage of the classical Bernstein method to obtain Lipschitz regularity (see \cite{LS2005} for instance).

In the specific case where $c\equiv 0$ and $\Omega$ is convex and bounded, the global Lipschitz estimate of the solution was obtained in \cite{GOS}.
See Remark \ref{rem:GOS}.
Moreover, a very interesting example was given in \cite{GOS} to show that the solution is not globally Lipschitz continuous if $\Omega$ is not convex.
Motivated by this example, we give two examples showing that $u$ is not globally Lipschitz continuous if we do not impose \eqref{condition:c}.
Furthermore, the examples demonstrate that condition \eqref{condition:c} is sharp.

Let us note that the graph mean curvature flow with the Neumann boundary conditions has been studied much in the literature (see \cite{H, Guan, MWW} and the references therein). 

\medskip

We next study the large time behavior of $u$ under condition  \eqref{condition:c}.
\begin{theorem}\label{thm:asym}
Assume \eqref{condition:c}.
Let $u$ be the unique viscosity solution to \eqref{eq:levelset1}--\eqref{eq:levelset3}.
Then,
\[
u(\cdot,t)\rightarrow v,\quad \text{ as } t\rightarrow \infty,
\]
uniformly on $\overline{\Omega}$ for some Lipschitz function $v$, which is a viscosity solution to 
\begin{equation}\label{eq:stationary}
\begin{cases}
-\left(\div \left(\frac{D v}{|D v|}\right)+c(x)\right)|D v|=0 \quad \ &{\rm in } \ \Omega,\\
\displaystyle \frac{\partial v}{\partial \Vec{\mathbf{n}}}=0\quad \ &{\rm on } \ \partial\Omega.
\end{cases}
\end{equation}
\end{theorem}
We prove Theorem \ref{thm:asym} by using a  Lyapunov function, which is quite standard. 
We say that $v$ is the large time profile of the solution $u$.
It is important to note that the stationary problem \eqref{eq:stationary} may have various different solutions, and thus, the question on how the large time profile $v$ depends on the initial data $u_0$ is rather delicate and challenging. 
We are able to answer this question in the radially symmetric setting, and it is still widely open in the general settings. 

\begin{theorem}\label{thm:radlimit}
Suppose that, by abuse of notions,
\begin{equation}\label{con:radial}
\begin{cases}
\Omega = B(0,R) \text{ for some } R>0,\\
c(x) = c(r) \text{ for } |x|=r \in [0, R],\\
u_0(x)= u_0(r)  \text{ for } |x|=r \in [0, R].
\end{cases}
\end{equation} 
Here, $c \in C^1([0,R], [0,\infty))$, and $u_0 \in C^2([0,R])$ with $u_0'(R)=0$.
Denote by
\begin{align*}
&\mathcal{A}:=\left\{r\in(0,R]\,:\,c(r)=\frac{n-1}{r}\right\},\\
 &\mathcal{A}_{+}:=\left\{r\in(0,R]\,:\,c(r)>\frac{n-1}{r}\right\},\\
 &\mathcal{A}_{-}:=\left\{r\in(0,R]\,:\,c(r)<\frac{n-1}{r}\right\}.
\end{align*}
Define $d:(0,R] \to (0,R]$ as
\[
d(r)=
\begin{cases}
r \qquad &\text{ if } r \in \cA,\\
\max \left (\cA \cap (0,r) \right) \qquad &\text{ if } r \in \cA_+,\\
\min \left (\cA \cap (r,R] \right) \qquad &\text{ if } r \in \cA_- \text{ and } \cA \cap (r,R] \neq \emptyset,\\
R \qquad &\text{ if } r \in \cA_- \text{ and } \cA \cap (r,R] = \emptyset.
\end{cases}
\]
Write $u(x,t)=\phi(|x|,t)$ for $x\in \Omega=B(0,R)$ and $t\geq 0$.
Then, the limiting profile $\phi_{\infty}(r)=\lim_{t\to\infty}\phi(r,t)$ can be written in terms of $u_0$ as:
for each $r_0 \in (0,R]$,

\begin{equation}\label{radlimit}
\phi_{\infty}(r_0)=\max\left\{ u_0(r)\,:\, r \geq d(r_0) \right\}.
\end{equation}

\end{theorem}

As a by-product, Theorem \ref{thm:radlimit}  shows that the solution to \eqref{eq:levelset1}--\eqref{eq:levelset3} is not globally Lipschitz continuous with an appropriate choice of initial data $u_0$. 

\begin{corollary}\label{thm:non-Lip-rad}
Consider the setting in Theorem  {\rm\ref{thm:radlimit}}.
Assume that there exist $0<a<b<R$ such that $a,b \in \cA$ and $(a, b) \subset \cA_-$.
Assume further that $u_0$ is a $C^2$ function on $[0,R]$ such that
\[
u_0(r)=
\begin{cases}
1 \quad &\text{ for } r\leq a,\\
\in (0,1) \quad &\text{ for } a < r \leq b,\\
0 \quad &\text{ for } b < r \leq R.
\end{cases}
\]
Then, $u$ is not globally Lipschitz, and
\[
\phi_\infty(r)=
\begin{cases}
1 \quad &\text{ for } r\leq a,\\
0 \quad &\text{ for } a< r \leq R.
\end{cases}
\]
\end{corollary}

Lastly, we give another example to show the non global Lipschitz phenomenon in Theorem \ref{thm:conv}.
Since we deal with the situation where $\Omega$ is unbounded there, we leave the precise statement of Theorem \ref{thm:conv} and corresponding adjustments to Section \ref{sec:blow-up2}.

\medskip

Our problem \eqref{eq:levelset1}--\eqref{eq:levelset2} basically describes a  level-set forced mean curvature flow with the homogeneous Neumann boundary condition.
If a level set of the unknown $u$ is a smooth enough surface, then it evolves with the normal velocity $V(x)=\kappa + c(x)$, where $\kappa$ equals $(n-1)$ times the mean curvature of the surface at $x$, and it perpendicularly intersects $\partial \Omega$ (if ever).
What is really interesting and delicate here is the competition between the forcing term $c(x)$ and the mean curvature term $\kappa$ coupled with the constraint on perpendicular intersections of the level sets with the boundary.
It is worth emphasizing that we do not assume $\Omega$ is convex, and the geometry of $\partial \Omega$ plays a crucial role in the behavior of the solution here. Indeed, analyzing the competition between the two constraints, the force and the boundary condition subjected to $\partial\Omega$, as time evolves in viscosity sense is the main topic of this paper.

We now briefly describe our approaches to get the aforementioned results.
 We use the maximum principle and rely on the classical Bernstein method to establish \emph{a priori} gradient estimates for the solution. 
 The main difficulty is when a maximizer is located on the boundary, which we cannot apply the maximum principle directly. We deal with this difficulty by considering a multiplier that puts the maximizer, with the homogeneous Neumann boundary condition, inside the domain so that the maximum principle is applicable. 
 To the best of our knowledge, the idea of handling a maximizer in the proof of Theorem \ref{thm:global-grad} for the level-set equation for forced mean curvature flows under the Neumann boundary condition is new in the literature.

Once we get a global Lipschitz estimate for the solution, by using a standard Lyapunov function, we prove the convergence in Theorem \ref{thm:asym}. Next, the radially symmetric setting is considered, and \eqref{eq:levelset1}--\eqref{eq:levelset3} are reduced to a first-order singular Hamilton-Jacobi equation with the homogeneous Neumann boundary condition; see \cite{GMT, GMT2} for a related problem on the whole space. 
By using the representation formula for the Neumann problem (see, e.g., \cite{I2}), we are able to obtain Theorem \ref{thm:radlimit} and Corollary \ref{thm:non-Lip-rad}. 
The situation considered in Theorem \ref{thm:conv} is related to that in \cite[Section 4]{Ohtsuka} with no forcing term. 
As we have a constant forcing $c$ interacting with the boundary, the construction in the proof of Theorem \ref{thm:conv} is rather delicate and involved. 
It is worth emphasizing that Corollary \ref{thm:non-Lip-rad} and Theorem \ref{thm:conv} demonstrate that  condition  \eqref{condition:c}, which is needed for the global Lipschitz regularity of $u$, is essentially optimal.

\medskip
We conclude this introduction by giving a non exhaustive list of  related works to our paper.  
There are several asymptotic analysis results on the forced mean curvature flows with Neumann boundary conditions \cite{Guan, MNL, MT, Xu} or with periodic boundary conditions \cite{CN}, but they are all for graph-like surfaces. 
The volume preserving mean curvature flow, which is a different type of forced mean curvature flows, was studied in 
\cite{KKb, KKP}. 
Recently, the relation between the level set approach and the varifold approach for \eqref{eq:levelset1} with $c\equiv0$ was investigated in \cite{Aimi}. 
We also refer to \cite{GMT2, Hamamuki-Misu} for some recent results on the asymptotic growth speed of solutions to forced mean curvature flows with discontinuous source terms in the whole space. 

\medskip
\subsection*{Organization of the paper}
The paper is organized as follows. 
In Section  \ref{sec:prelim}, we give the notion of viscosity solutions to the problem and some basic results.
In Section \ref{sec:Lip}, we prove the local and global gradient estimates.
Section \ref{sec:largetime} is devoted to the study on large time behavior of the solution and its large time profile.
We give two examples that the spatial gradient of the solution grows to infinity as time tends to infinity in Sections \ref{sec:blow-up} and \ref{sec:blow-up2} if we do not impose assumption \eqref{condition:c}  on the force $c$.


\section{Preliminaries}\label{sec:prelim}
In this section, we recall the notion of viscosity solutions to the Neumann boundary problem \eqref{eq:levelset1}--\eqref{eq:levelset3} and give some related results.

\medskip

Let $\mathcal{S}^n$ be the set of symmetric matrices of size  $n$. 
 Define $F:\overline{\Omega}\times(\mathbb{R}^n \setminus \{0\})\times \mathcal{S}^n \rightarrow \mathbb{R}$ by
$$
F(x,p,X)=\text{trace}\left(\left(I-\frac{p\otimes p}{|p|^2}\right)X\right)+c(x)|p|.
$$
We denote the semicontinuous envelopes of $F$ by, for $(x,p,X)\in\overline{\Omega}\times \mathbb{R}^n \times\mathcal{S}^n$,
$$
F_*(x,p,X)=\liminf_{(y,q,Y)\rightarrow(x,p,X)}F(y,q,Y),
\quad 
F^*(x,p,X)=\limsup_{(y,q,Y)\rightarrow(x,p,X)}F(y,q,Y).
$$

\begin{definition} An upper semicontinuous function $u:\overline{\Omega}\times [0,\infty)\rightarrow \mathbb{R}$ is said to be a viscosity subsolution of \eqref{eq:levelset1}--\eqref{eq:levelset3}   
if $u(\cdot,0)\leq u_0$ on $\overline{\Omega}$, and, 
for any $\varphi\in C^2(\overline{\Omega}\times[0,\infty))$, 
if $(\hat{x},\hat{t})\in \overline{\Omega}\times(0,\infty)$ is a maximizer of $u-\varphi$, and if $\hat{x}\in\Omega$, then 
$$
\varphi_t(\hat{x},\hat{t})-F^*(\hat{x},D\varphi(\hat{x},\hat{t}),D^2\varphi(\hat{x},\hat{t}))\leq 0;
$$
if $\hat{x}\in\partial\Omega$, then 
$$
\min\left\{\varphi_t(\hat{x},\hat{t})-F^*(\hat{x},D\varphi(\hat{x},\hat{t}),D^2\varphi(\hat{x},\hat{t})),\frac{\partial\varphi}{\partial\Vec{\mathbf{n}}}(\hat{x},\hat{t})\right\}\leq 0.
$$
Similarly, a lower semicontinuous function $u:\overline{\Omega}\times [0,\infty)\rightarrow \mathbb{R}$ is said to be a viscosity supersolution of \eqref{eq:levelset1}--\eqref{eq:levelset3}  if $u(\cdot,0)\geq u_0$ on $\overline{\Omega}$, and, 
for any $\varphi\in C^2(\overline{\Omega}\times[0,\infty))$, 
if $(\hat{x},\hat{t})\in \overline{\Omega}\times(0,\infty)$ is a minimizer of $u-\varphi$, and if $\hat{x}\in\Omega$, then 
$$
\varphi_t(\hat{x},\hat{t})-F_*(\hat{x},D\varphi(\hat{x},\hat{t}),D^2\varphi(\hat{x},\hat{t}))\geq 0;
$$
if $\hat{x}\in\partial\Omega$, then 
$$
\max\left\{\varphi_t(\hat{x},\hat{t})-F_*(\hat{x},D\varphi(\hat{x},\hat{t}),D^2\varphi(\hat{x},\hat{t})),\frac{\partial\varphi}{\partial\Vec{\mathbf{n}}}(\hat{x},\hat{t})\right\}\geq 0.
$$

\noindent Finally, a continuous function $u$ is said to be a viscosity solution of \eqref{eq:levelset1}--\eqref{eq:levelset3} if $u$ is both its viscosity subsolution and its viscosity supersolution.
\end{definition}

Henceforth, since we are always concerned with viscosity solutions, the adjective ``viscosity" is omitted.
The following comparison principle for solutions to  \eqref{eq:levelset1}--\eqref{eq:levelset3} in a bounded domain is well known (see, e.g., \cite{GS}).

\begin{proposition}[Comparison principle for  \eqref{eq:levelset1}--\eqref{eq:levelset3}]\label{prop:comp}
Let $u$ and $v$ be a subsolution and a supersolution of  \eqref{eq:levelset1}--\eqref{eq:levelset3}, respectively. 
Then, $u\leq v$ in $\overline{\Omega}\times[0,\infty)$.
\end{proposition}

To obtain Lipschitz estimates, it is convenient to consider an approximate problem of  \eqref{eq:levelset1}--\eqref{eq:levelset3}  by considering, for $\varepsilon>0$, $T>0$,

\begin{equation}\label{eq:levelset-ep}
\begin{cases}
u^{\varepsilon}_t=\sqrt{\varepsilon^2+|D u^{\varepsilon}|^2}\,\text{div}\left(\frac{D u^{\varepsilon}}{\sqrt{\varepsilon^2+|D u^{\varepsilon}|^2}}\right)+c(x)\sqrt{\varepsilon^2+|D u^{\varepsilon}|^2} \quad &\text{ in } \Omega\times(0,T],\\
\displaystyle \frac{\partial u^{\varepsilon}}{\partial \Vec{\mathbf{n}}}=0 \quad &\text{ on } \partial\Omega\times[0,T],\\
u^{\varepsilon}(x,0)=u_0(x) \quad &\text{ on } \overline{\Omega}.
\end{cases}
\end{equation}
Equation \eqref{eq:levelset-ep} describes the motion of the graph of $\frac{u^\ep}{\ep}$ under the forced mean curvature flow $V=\kappa+c$ in $\Omega$ with right contact angle condition on $\partial \Omega$.
The following result on a priori estimates on the gradient of $u^\ep$ plays a crucial role in our analysis.

\begin{theorem}[A priori estimates]\label{thm:local-grad-ep}
Assume that $\partial \Omega$ is smooth and $c\in C^\infty(\ol \Omega)$.
For each $\ep \in (0,1)$  and  $T>0$, assume that $u^\ep \in C^\infty(\ol\Omega \times (0,T]) \cap C^1(\overline{\Omega}\times[0,T])$ is the unique solution of \eqref{eq:levelset-ep}.
Then, there exist a constant $M>0$ and a constant $C_T>0$ depending on $T$ such that 
\begin{equation}\label{ineq:apriori}
\lVert u^{\varepsilon}_t\rVert_{L^{\infty}(\overline{\Omega}\times[0,T])}\leq M \quad \text{ and } \quad
\lVert Du^{\varepsilon}\rVert_{L^{\infty}(\overline{\Omega}\times[0,T])}\leq C_T.
\end{equation}
Here, $M$ and $C_T$ are independent of $\ep \in (0,1)$.
\end{theorem}

The proof of Theorem \ref{thm:local-grad-ep} is given in the next section.
The a priori estimates then allow us to get the existence and uniqueness of solutions to \eqref{eq:levelset-ep}.

\begin{proposition}\label{prop:exuniquegra}
For each $\ep \in (0,1)$  and  $T>0$, equation \eqref{eq:levelset-ep} has a unique continuous solution $u^{\varepsilon}$.
Furthermore, $u^{\varepsilon}\in  C^{2,1}(\Omega\times(0,T]) \cap C^1(\overline{\Omega}\times[0,T])$ and \eqref{ineq:apriori} holds.
\end{proposition}
Proposition \ref{prop:exuniquegra} can be obtained by the classical parabolic PDE theory. 
For instance, 
we refer to \cite{MT} for a similar form of Proposition \ref{prop:exuniquegra}. 
The proof of this proposition is quite standard, and hence, is omitted here. 

Once we get \eqref{ineq:apriori},
by the standard stability result of viscosity solutions, and the uniqueness of viscosity solutions to \eqref{eq:levelset1}--\eqref{eq:levelset3}, we imply that
\begin{equation*}
\label{eq:uni}
u^{\varepsilon}\rightarrow u\quad \text{as} \ \varepsilon\rightarrow 0 \quad
\text{uniformly on} \ \overline{\Omega}\times[0,T] 
\end{equation*}
for each $T>0$.
Moreover, Theorem \ref{thm:local-grad-ep} and Proposition \ref{prop:exuniquegra} give us right away Theorem \ref{thm:local-grad}.


\section{Lipschitz regularity} \label{sec:Lip}

In this section, we prove Theorems \ref{thm:local-grad}, \ref{thm:global-grad}, and \ref{thm:local-grad-ep}. 
As noted, it is actually enough to prove Theorems \ref{thm:global-grad} and \ref{thm:local-grad-ep}. 
First, we prove that the time derivative of $u^\ep$ is bounded.

\begin{lemma}\label{lem:time-grad}
Assume that $\partial \Omega$ is smooth and $c\in C^\infty(\ol \Omega)$.
Suppose that $u^\ep \in C^\infty(\ol \Omega \times (0,T]) \cap C^1(\overline{\Omega}\times[0,T])$ is the unique  solution of \eqref{eq:levelset-ep} for each $\varepsilon \in (0,1)$ and $T>0$. 
Then, there exists $M>0$ depending only on the forcing term $c$ and the initial data $u_0$ such that, for $\ep \in (0,1)$,
$$
\lVert u^{\varepsilon}_t\rVert_{L^{\infty}(\overline{\Omega}\times[0,T])}\leq\lVert u^{\varepsilon}_t(\cdot,0)\rVert_{L^{\infty}(\overline{\Omega})} \leq M.
$$
\end{lemma}

\begin{proof}
Set $b(p)=I_n-p\otimes p/(\varepsilon^2+|p|^2)$. 
Then \eqref{eq:levelset-ep} is expressed as
\begin{equation}\label{eq:levelset-prime}
u^{\varepsilon}_t-b^{ij}(Du^{\varepsilon})u^{\varepsilon}_{ij}-c(x)\sqrt{\varepsilon^2+|Du^{\varepsilon}|^2}=0 \quad\text{ in } \Omega\times(0,T].
\end{equation}
Here, we use the Einstein summation convention, 
and we write 
$f_i=\frac{\partial f}{\partial x_i}$ and
$f_{ij}=\frac{\partial^2 f}{\partial x_i \partial x_j}$
for $i,j=1,\ldots,n$, where $f=f(x,t)$ is a given function.
We now show that
\begin{equation}\label{eq:t-0}
\lVert u^{\varepsilon}_t\rVert_{L^{\infty}(\overline{\Omega}\times[0,T])}\leq\lVert u^{\varepsilon}_t(\cdot,0)\rVert_{L^{\infty}(\overline{\Omega})}.
\end{equation}
To prove \eqref{eq:t-0}, it is enough to obtain the upper bound
\[
\max_{\overline{\Omega}\times[0,T]} u^{\varepsilon}_t = \max_{\overline{\Omega}} u^{\varepsilon}_t (\cdot,0)
\]
as the lower bound can be obtained analogously.

Differentiating \eqref{eq:levelset-prime} with respect to $t$ yields
$$
(u^{\varepsilon}_t)_t-b^{ij}(u^{\varepsilon}_t)_{ij}-(b^{ij})_tu^{\varepsilon}_{ij}-c(x)\frac{(u^{\varepsilon}_t)_lu^{\varepsilon}_l}{\sqrt{\varepsilon^2+|Du^{\varepsilon}|^2}}=0,
$$
where
$$
(b^{ij})_t=-\frac{(u^{\varepsilon}_t)_iu^{\varepsilon}_j}{\varepsilon^2+|Du^{\varepsilon}|^2}-\frac{u^{\varepsilon}_i(u^{\varepsilon}_t)_j}{\varepsilon^2+|Du^{\varepsilon}|^2}+\frac{2u^{\varepsilon}_iu^{\varepsilon}_ju^{\varepsilon}_l(u^{\varepsilon}_t)_l}{(\varepsilon^2+|Du^{\varepsilon}|^2)^2}.
$$

Suppose, on the contrary, that $u^{\varepsilon}_t(x,t)>\max_{\overline{\Omega}} u^\ep_t(\cdot,0)$ for some $(x,t)\in\overline{\Omega}\times(0,T]$. 
Then, there exist a small number $\delta>0$ and $(x_0,t_0)\in\overline{\Omega}\times(0,T]$ such that $(x_0,t_0)\in\argmax_{\overline{\Omega}\times(0,T]}(u^{\varepsilon}_t-\delta t)$.

At $(x_0,t_0)$, we have $Du^{\varepsilon}_t=0,$ and note that the boundary case $x_0\in\partial\Omega$ is included due to the homogeneous Neumann boundary condition. 
Thus,
\begin{equation}\label{max-t}
(u^{\varepsilon}_t)_t-b^{ij}(u^\ep_t)_{ij}=0,\quad\text{at}\ (x_0,t_0).
\end{equation}
On the other hand, $(u^\ep_t-\delta t)_t\geq0,$ $-b^{ij}(u^{\varepsilon}_t)_{ij}\geq0$ at $(x_0,t_0).$
 Note that the Neumann boundary condition is used for $D^2u^{\varepsilon}_t \leq0$ at $(x_0,t_0)$ as well. 
 Since $(u^\ep_t)_t\geq\delta>0$, we arrive at a contradiction in \eqref{max-t}.
 Thus, \eqref{eq:t-0} holds.
Choose
\[
M = n^2 \lVert D^2u_0\rVert_{L^\infty(\ol \Omega)} + \lVert c \sqrt{1+|Du_0|^2} \rVert_{L^\infty(\ol \Omega)} 
\]
to complete the proof.
\end{proof}

We are now ready to prove Theorems \ref{thm:global-grad} and \ref{thm:local-grad-ep} using the classical Bernstein method.
It is important emphasizing that the boundary behavior needs to be handled rather carefully. 
We first give a proof of Theorem \ref{thm:global-grad}.  

\begin{proof}[Proof of Theorem~\ref{thm:global-grad}] 
Assume first that $\partial \Omega$ is smooth and $c\in C^\infty(\ol \Omega)$.
For each $\ep \in (0,1)$  and  $T>0$, let $u^\ep \in C^\infty(\ol\Omega \times (0,T]) \cap C^1(\overline{\Omega}\times[0,T])$ be the unique solution of \eqref{eq:levelset-ep}.

Let $w^{\varepsilon}=\sqrt{\varepsilon^2+|Du^{\varepsilon}|^2}$. 
In view of Lemma \ref{lem:time-grad}, we only need to show that
\begin{equation}\label{eq:bound-w-ep}
\max_{\ol \Omega \times [0,T]} w^\ep \leq C
\end{equation}
for some positive constant $C$ depending only on $\|u_0\|_{C^2(\ol \Omega)}$, $\|c\|_{C^1(\ol \Omega)}$,  the constants $n$,  $C_0$, $K_0$, and $\delta$ from \eqref{condition:c}.
The crucial point here is $C$ does not depend on $T$ and $\ep$.
Fix $(x_0,t_0)\in\argmax_{\overline{\Omega}\times[0,T]}w^{\varepsilon}$. 
If $t_0=0$, then
\[
\max_{\ol \Omega \times [0,T]} w^\ep \leq w^\ep(x_0,0) \leq \|Du_0\|_{L^\infty(\ol \Omega)} + 1,
\]
and \eqref{eq:bound-w-ep} is valid.
We next consider the case $t_0>0$.

We write $u=u^{\varepsilon}$, $w=w^{\varepsilon}$ in this proof for brevity.
Differentiate \eqref{eq:levelset-prime} in $x_k$ and multiply the result by $u_k$ to get
$$
u_ku_{kt}-(D_pb^{ij}\cdot Du_k)u_ku_{ij}-b^{ij}u_{k}u_{kij}-u_k c_k w-c\frac{u_ku_{lk}u_l}{w}=0.
$$
Substituting $ww_t=u_ku_{kt},\ ww_k=u_lu_{kl}$ and $ww_{ij}=u_{kij}u_k+b^{kl}u_{ki}u_{lj},$ we get
\begin{equation}\label{eq:w}
ww_t-w(D_pb^{ij}\cdot Dw)u_{ij}-wb^{ij}w_{ij}+b^{ij}b^{kl}u_{ki}u_{lj}-wDu\cdot Dc-cDu\cdot Dw=0. 
\end{equation}
We divide the proof into two cases: $x_0\in\Omega$ and $x_0\in\partial\Omega$.

\medskip

\noindent {\bf Case 1: the interior case $x_0\in\Omega$.}
We follow the computations of \cite[Lemma 4.1]{GMOT}. 
At $(x_0,t_0)$, we have $w_t\geq0$, $Dw=0$, $D^2w\leq0$, and thus
$$
wDu\cdot Dc\geq b^{ij}b^{kl}u_{ki}u_{lj}.
$$
We then use the Cauchy-Schwarz inequality
$$
(\textrm{tr}\alpha\beta)^2\leq\textrm{tr}(\alpha^2)\textrm{tr}(\beta^2)
$$
for all $\alpha, \beta\in\mathcal{S}^n$, and put $\alpha=A^{\frac{1}{2}}BA^{\frac{1}{2}}$, $\beta=I_n$, 
where $A=(b^{ij})$, $B=(u_{kl})$, $I_n$ the $n$ by $n$ identity matrix to get $\textrm{tr}(AB)^2\geq(\textrm{tr}AB)^2/\textrm{tr}(I_n).$

Therefore, at $(x_0,t_0)$,
\begin{align*}
\left|Dc(x_0)\right|w^2\geq wDu\cdot Dc\geq b^{ij}b^{kl}u_{ki}u_{lj}=\textrm{tr}(AB)^2
\geq\frac{(\textrm{tr}AB)^2}{\textrm{tr}(I_n)}=\frac{1}{n}\left(u_t-c(x_0)w\right)^2
\end{align*}
Since $\frac{1}{n}c(x)^2-\left|Dc(x)\right|\geq\delta>0$ by \eqref{condition:c}, we imply that at $(x_0,t_0)$,
\[
\delta w^2 \leq \frac{2u_t c(x_0)}{n}  w \quad \Longrightarrow \quad w \leq \frac{2 M \|c\|_{L^\infty(\ol \Omega)}}{n\delta},
\]
which confirms \eqref{eq:bound-w-ep}.

\medskip

\noindent {\bf Case 2: the boundary case $x_0\in\partial\Omega$.}
As $\partial \Omega$ is $C^{2,\theta}$, we assume that $\mathbf{n}$ is defined as a $C^1$ function in a neighborhood of $\partial \Omega$.
Note that the Neumann boundary condition $Du\cdot \Vec{\mathbf{n}}=0$ gives $\left(D^2u\, \Vec{\mathbf{n}}+D\Vec{\mathbf{n}} Du\right)\cdot  v =0$ for all $ v \in \R^n$ perpendicular to $\Vec{\mathbf{n}}$  on $\partial \Omega \times [0,T]$.
Thus, on $\partial \Omega \times [0,T]$,
\[
\frac{\partial w}{\partial \Vec{\mathbf{n}}}=\frac{D^2u Du}{w}\cdot\Vec{\mathbf{n}}=-\frac{D\Vec{\mathbf{n}}Du\cdot Du}{w}\leq C_0\frac{|Du|^2}{w},
\]
where $C_0=\sup\{-\lambda\,:\,\lambda\ \textrm{is a principal curvature of $\partial\Omega$ at $x_0$ for $x_0 \in \partial \Omega$}\}$.

If $C_0<0$, then $\frac{\partial w}{\partial \Vec{\mathbf{n}}}<0$ on $\partial \Omega \times [0,T]$, and hence $w$ cannot attain its maximum on $\partial \Omega \times [0,T]$. 
Therefore, $C_0\geq0$. 
We consider the case when $C_0>0$ first, and deal with the case when $C_0=0$ later. We note that if $C_0>0$, then
\[
\frac{\partial w}{\partial \Vec{\mathbf{n}}}\leq C_0\frac{|Du|^2}{w}<C_0w.
\]

\medskip

Take $x_c\in\Omega$ so that $B:=B(x_c, K_0/2)$ is inside $\Omega$ and tangent to the boundary $\partial\Omega$ at $x_0$. 
Consider a multiplier 
\[
\rho(x)=-\frac{C_0}{K_0}|x-x_c|^2+\frac{C_0 K_0}{4}+1 \quad \text{ for } x \in \ol \Omega.
\]
Then, $\rho>1$ in $B$, $\rho=1$ on $\partial B$, and $\rho \leq 1$ on $\ol \Omega \setminus B$.
Besides, $C_0 \rho(x_0)+\frac{\partial \rho}{\partial \Vec{\mathbf{n}}}(x_0)=0$. 

Denote by $\psi=\rho w$.
Then, at $(x_0,t_0)$,
\begin{equation}\label{rho-1}
\frac{\partial \psi}{\partial \Vec{\mathbf{n}}}=\frac{\partial (\rho w)}{\partial \Vec{\mathbf{n}}}=\rho\frac{\partial w}{\partial \Vec{\mathbf{n}}}+w\frac{\partial \rho}{\partial \Vec{\mathbf{n}}}<w\left(C_0\rho+\frac{\partial \rho}{\partial \Vec{\mathbf{n}}}\right) = 0.
\end{equation}
By the choice of $\rho$, it is clear that
\begin{equation*}\label{rho-2}
 \psi(z,t) \leq w(z,t)\leq w(x_0,t_0) =  \psi(x_0,t_0) \quad\text{ for } (z,t) \in \left( \ol \Omega \setminus B \right)\times [0, T],
\end{equation*}
and, by \eqref{rho-1},
\begin{equation}\label{rho-3}
\max_{\overline{\Omega}\times[0,T]}\rho w = \max_{\overline{B}\times[0,T]}\rho w > \psi(x_0,t_0)=w(x_0,t_0).
\end{equation}
Let $(x_1,t_1)\in\argmax_{\overline{\Omega}\times[0,T]}\rho w$.
 If $t_1=0$, then for all $(x,t)\in\overline{\Omega}\times[0, T]$,
\begin{align*}
w(x,t)&\leq w(x_0,t_0)=\rho(x_0)w(x_0,t_0)\leq\rho(x_1)w(x_1,0) \\
&\leq  \left(\frac{C_0 K_0}{4}+1 \right)\left(\|Du_0\|_{L^\infty(\ol \Omega)} + 1\right),
\end{align*}
and we are done. 
Thus, we may assume that $t_1>0$. 
In light of \eqref{rho-1}--\eqref{rho-3}, we yield that $x_1\in B \subset \Omega$. 
At this point $(x_1,t_1)$, we have $\psi_t\geq0,$ $D\psi=0,$ $D^2\psi\leq0$. 
Consequently, as $\psi_t=\rho_t w+\rho w_t$, $D\psi=wD\rho+\rho Dw$, and $\psi_{ij}=w_{ij}\rho+w_i\rho_j+w_j\rho_i+w\rho_{ij}$, we have at $(x_1,t_1)$,

\[
\begin{cases}
w_t\geq-\frac{\rho_t}{\rho}w=0,\\
 Dw=-\frac{w}{\rho}D\rho,\\
w_{ij}=\frac{1}{\rho}(\psi_{ij}-w_i\rho_j-w_j\rho_i-w\rho_{ij}).
\end{cases}
\]

\noindent Therefore, at $(x_1,t_1)$, by \eqref{eq:w}
\begin{multline*}
-\frac{\rho_t}{\rho}w^2+\frac{w^2}{\rho}(D_{p}b^{ij}\cdot D\rho)u_{ij}+\frac{w}{\rho}b^{ij}(w_i\rho_j+w_j\rho_i+w\rho_{ij})\\+b^{ij}b^{kl}u_{ki}u_{lj}-wDu\cdot Dc+\frac{cw}{\rho}Du\cdot D\rho\leq0.
\end{multline*}
Now,
\[
b^{ij}_{p_l}=-\frac{\delta_{il}u_j}{\varepsilon^2+|Du|^2}-\frac{\delta_{jl}u_i}{\varepsilon^2+|Du|^2}+\frac{2u_iu_ju_l}{(\varepsilon^2+|Du|^2)^2},
\]
and thus,
\begin{align*}
w(D_pb^{ij}\cdot D\rho)u_{ij}&=w\left(-\frac{\rho_iu_ju_{ij}}{\varepsilon^2+|Du|^2}-\frac{\rho_ju_iu_{ij}}{\varepsilon^2+|Du|^2}+\frac{2u_iu_ju_l\rho_lu_{ij}}{(\varepsilon^2+|Du|^2)^2}\right)\\
&=-2Dw\cdot D\rho+\frac{2(Du\cdot D\rho)(Du\cdot Dw)}{w^2}.
\end{align*}
Hence,
\begin{align*}
&w(D_pb^{ij}\cdot D\rho)u_{ij}+b^{ij}w_i\rho_j+b^{ij}w_j\rho_i\\
=\ &\frac{2(Du\cdot D\rho)(Du\cdot Dw)}{w^2}-\frac{u_iu_jw_i\rho_j}{w^2}-\frac{u_iu_jw_j\rho_i}{w^2}=0.
\end{align*}
All in all, at $(x_1,t_1)\in\argmax_{\overline{\Omega}\times(0,T]}\rho w$ with $x_1\in B \subset \Omega$, the inequality
\begin{equation}\label{rho-4}
-\frac{\rho_t}{\rho}w^2+\frac{\rho_{ij}}{\rho}b^{ij}w^2+b^{ij}b^{kl}u_{ki}u_{lj}-wDu\cdot Dc+\frac{cw}{\rho}Du\cdot D\rho\leq0
\end{equation}
holds.
Note that $\rho_t=0$ here, but we keep this term in the above formula for the usage in the proof of Theorem \ref{thm:local-grad-ep} later.

Using the Cauchy-Schwarz type inequality as in the above, we obtain
\begin{align*}
\frac{1}{n}(u_t-c(x_1)w)^2\leq b^{ij}b^{kl}u_{il}u_{kj}&\leq-\frac{w^2}{\rho}b^{ij}\rho_{ij}+wDu\cdot Dc-\frac{cw}{\rho}Du\cdot D\rho\\
&\leq \frac{2C_0}{K_0}\frac{w^2}{\rho} \left(n-\frac{|Du|^2}{\varepsilon^2+|Du|^2}\right)+|Dc|w^2+C_0|c|w^2\\
&\leq\left(\frac{2nC_0}{K_0}+|Dc(x_1)|+C_0|c(x_1)|\right)w^2.
\end{align*}
By \eqref{condition:c}, 
\[
\frac{1}{n}c(x)^2-|Dc(x)|-C_0 |c(x)|-\frac{2n C_0}{K_0}\geq\delta>0\quad\quad\text{for all}\ x\in\overline{\Omega}
\]
for some $\delta>0$, we see that $w(x_1,t_1) \leq \frac{2 M \|c\|_{L^\infty(\ol \Omega)}}{n\delta}$. 
Thus,
\[
w(x_0,t_0) \leq \rho(x_1) w(x_1,t_1) \leq \left(\frac{C_0 K_0}{4}+1\right) \frac{2 M \|c\|_{L^\infty(\ol \Omega)}}{n\delta}.
\]

Now, we handle the case when $C_0=0$. We consider a multiplier
\[
\rho(x)=-\frac{\delta_1}{K_0}|x-x_c|^2+\frac{\delta_1 K_0}{4}+1 \quad \text{ for } x \in \ol \Omega,
\]
where 
\[
\delta_1=\frac{\delta}{2(\|c\|_{L^\infty} + \frac{2n}{K_0})}>0.
\]
Then, at $(x_0,t_0)$,
\[
\frac{\partial w}{\partial \Vec{\mathbf{n}}}\leq C_0\frac{|Du|^2}{w} = 0,
\]
and
\[
\frac{\partial \psi}{\partial \Vec{\mathbf{n}}}=\frac{\partial (\rho w)}{\partial \Vec{\mathbf{n}}}=\rho\frac{\partial w}{\partial \Vec{\mathbf{n}}}+w\frac{\partial \rho}{\partial \Vec{\mathbf{n}}}\leq w\frac{\partial \rho}{\partial \Vec{\mathbf{n}}}<0.
\]
Following the same argument as above with $\delta_1$ in place of $C_0$, we see that
\begin{align*}
\frac{1}{n}(u_t-c(x_1)w)^2 \leq\left(\frac{2n\delta_1}{K_0}+|Dc(x_1)|+C_0|c(x_1)|\right)w^2.
\end{align*}
This inequality, together with the fact that
\[
\frac{1}{n}c(x)^2-|Dc(x)|-\delta_1 |c(x)|-\frac{2n \delta_1}{K_0}\geq \delta-\frac{1}{2}\delta = \frac{1}{2}\delta >0\quad\quad\text{for all}\ x\in\overline{\Omega},
\]
implies \eqref{eq:bound-w-ep}.

By \eqref{eq:bound-w-ep} and Lemma~\ref{lem:time-grad}, $Du^{\varepsilon}$ and $u^{\varepsilon}_t$ are uniformly bounded in $\Omega \times [0,T]$ for all $\varepsilon \in (0,1)$ and $T>0$. 
Note that the bound depends only on $\|u_0\|_{C^2(\ol \Omega)}$, $\|c\|_{C^1(\ol \Omega)}$, the constants $n$,  $C_0$, $K_0$, and $\delta$ from \eqref{condition:c}.
By approximations, we see that the same result holds true in the case that $\partial \Omega \in C^{2,\theta}$ and $c\in C^{1}(\ol \Omega)$.
From the uniform convergence of $u^{\varepsilon}$ to the unique viscosity solution $u$ of \eqref{eq:levelset1}--\eqref{eq:levelset3}, we conclude that $u$ satisfies \eqref{eq:lip}.
\end{proof}

We remark for later usage that for any smooth function $\rho>0$, \eqref{rho-4} is valid at $(x_1,t_1)\in \argmax \left(\rho w\right) \cap \left(\Omega\times(0,T]\right)$.

\begin{remark} \label{rem:GOS}
Let us discuss a bit the case where $c\equiv 0$ and $\Omega$ is convex and bounded.
Then, $w$ satisfies
$$
ww_t-w(D_pb^{ij}\cdot Dw)u_{ij}-wb^{ij}w_{ij}+b^{ij}b^{kl}u_{ki}u_{lj}=0.
$$
And, on $\partial \Omega \times [0,T]$,
$$
\frac{\partial w}{\partial \Vec{\mathbf{n}}}=\frac{D^2u Du}{w}\cdot\Vec{\mathbf{n}}=-\frac{D\Vec{\mathbf{n}}Du\cdot Du}{w}\leq 0.
$$
By the usual maximum principle, we yield that
\[
\max_{\ol \Omega \times [0,T]} w = \max_{\ol \Omega} w(\cdot,0) \leq C.
\]
We thus recover the gradient bound in \cite{GOS}.
It is worth to note that in this specific situation, condition \eqref{condition:c} is not needed.
\end{remark}

\begin{proof}[Proof of Theorem \ref{thm:local-grad-ep}] 
Let $u=u^\ep$ and $w=\sqrt{\varepsilon^2+|Du^{\varepsilon}|^2}$ as in the proof of Theorem~\ref{thm:global-grad}.
As above, we may assume   $\partial \Omega$ is smooth and $c\in C^\infty(\ol \Omega)$.
Pick 
\[
M>\frac{2n (|C_0|+1)}{K_0}+\lVert Dc\rVert_{L^\infty(\ol \Omega)}+(|C_0|+1) \lVert c\rVert_{L^\infty(\ol \Omega)}
\]
and  $(x_0,t_0)\in\argmax_{\overline{\Omega}\times[0,T]}e^{-Mt}w(x,t)$.
If $t_0=0$, then we have that for $(x,t) \in \ol \Omega \times [0,T]$,
\[
w(x,t) \leq e^{MT}\left(\|Du_0\|_{L^\infty(\ol \Omega)} + 1\right).
\]
Consider next the case that $t_0>0.$ 
If $x_0\in\Omega$, then by \eqref{rho-4} with $\rho=e^{-Mt}$, at $(x_0,t_0)$, 
\[
Mw^2+b^{ij}b^{kl}u_{il}u_{kj}-wDu\cdot Dc\leq0.
\]
As $Mw^2-wDu\cdot Dc>0$ by the choice of $M$ and $b^{ij}b^{kl}u_{il}u_{kj}\geq0$, we arrive at a contradiction.
Thus, $x_0\in\partial\Omega$.

We repeat the proof of Theorem  \ref{thm:global-grad}.
Since $x_0 \in \argmax_{\ol \Omega} w(\cdot,t_0) \cap \partial \Omega$, we see as before that $C_0 \geq 0$.
We use a new multiplier
 \[
\rho(x,t)=e^{-Mt} \left(-\frac{C_0+1}{K_0}|x-x_c|^2+\frac{(C_0+1) K_0}{4}+1 \right) \quad \text{ for } (x,t) \in \ol \Omega \times [0,\infty).
\]
Here, $B=B(x_c, K_0/2)$ is inside $\Omega$ and tangent to the boundary $\partial\Omega$ at $x_0$. 

Put $w_M=e^{-Mt} w$ and note that $w_M(x_0,t_0)=\max_{\ol \Omega} w_M$, $\frac{\partial w_M}{\partial \Vec{\mathbf{n}}}\leq C_0w_M$ on $\partial \Omega \times [0,T]$, and
\[
\rho w =  \left(-\frac{C_0+1}{K_0}|x-x_c|^2+\frac{(C_0+1) K_0}{4}+1 \right)w_M.
\]
Observe as in the proof of Theorem  \ref{thm:global-grad} that  $\frac{\partial (\rho w)}{\partial \Vec{\mathbf{n}}}(x_0,t_0)<0$, $\rho w \leq w_M$ on $\left(\ol \Omega \setminus B \right) \times [0,T]$, and therefore, $\argmax \left(\rho w\right) \subset B \times [0,T]$.
Then, there is a point $(x_1,t_1)\in\argmax_{\overline{\Omega}\times[0,T]}\rho w$ with $(x_1,t_1)\in B\times [0,T]$.
Consider the case $t_1=0$.
For all $(x,t) \in \ol \Omega \times [0,T]$,
\begin{align*}
w_M(x,t) &\leq w_M(x_0,t_0) = (\rho w)(x_0,t_0) \leq (\rho w)(x_1,0)\\
&\leq  \left(\frac{(C_0+1) K_0}{4}+1 \right)\left(\|Du_0\|_{L^\infty(\ol \Omega)} + 1\right).
\end{align*}
Thus, for $(x,t) \in \ol \Omega \times [0,T]$,
\begin{equation}\label{eq:bound-w-final}
w(x,t) \leq e^{MT}\left(\frac{(C_0+1) K_0}{4}+1 \right)\left(\|Du_0\|_{L^\infty(\ol \Omega)} + 1\right).
\end{equation}
Next, we consider the case $t_1>0$.
At $(x_1,t_1)$, thanks to \eqref{rho-4}, we have
\[
Mw^2+\frac{\rho_{ij}}{\rho}b^{ij}w^2+b^{ij}b^{kl}u_{ki}u_{lj}-wDu\cdot Dc+\frac{cw}{\rho}Du\cdot D\rho \leq0.
\]
From this, recalling the choice of $M$, we obtain, as before,
\[
0 \leq b^{ij}b^{kl}u_{ki}u_{lj} \leq \left(-M +\frac{2n(C_0+1)}{K_0}+ |Dc(x_1)|+(C_0+1)|c(x_1)| \right) w^2 <0,
\]
which is absurd.
Thus, the case $t_1>0$ does not occur, and \eqref{eq:bound-w-final} holds true.
Lemma~\ref{lem:time-grad} and  \eqref{eq:bound-w-final}  then complete the proof.

\end{proof}

\begin{remark}
We note that Theorems \ref{thm:local-grad} and \ref{thm:global-grad} are still valid  when $\partial \Omega \in C^2$, $c\in C^1(\ol \Omega)$, and $u_0 \in C^2(\ol \Omega)$ by approximations as the Lipschitz bounds depend only on $\|u_0\|_{C^2(\ol \Omega)}$, $\|c\|_{C^1(\ol \Omega)}$, the constants $n$, $C_0$, $K_0$, and $T>0$ in case of Theorem \ref{thm:local-grad}, and $\delta$ from \eqref{condition:c} in case of Theorem \ref{thm:global-grad} .
\end{remark}


\section{Large time behavior of the solution} \label{sec:largetime} 
In this section, we prove the large time behavior of $u$, which is globally Lipschitz continuous thanks to Theorem \ref{thm:global-grad}. 
Let $L$ be the spatial Lipschitz constant of $u^\ep$ for $\ep \in (0,1)$ given by the proof of  Theorem \ref{thm:global-grad}. 

\begin{proof}[Proof of Theorem \ref{thm:asym}]
Although the proof is almost same as that of \cite[Theorem 1.2]{GTZ}, we give it for completeness.

We consider the following Lyapunov function
$$
I^{\varepsilon}(t)= \int_{\Omega}\sqrt{\varepsilon^2+|D u^{\varepsilon}|^2}\,dx.
$$
By calculation,
\begin{eqnarray*}
\frac{{\rm d}}{{\rm d}t}\int_{\Omega}\sqrt{\varepsilon^2+|D u^{\varepsilon}|^2}\,dx=\int_{\Omega}\frac{D u^{\varepsilon}\cdot D u_t^{\varepsilon}}{\sqrt{\varepsilon^2+|D u^{\varepsilon}|^2}}\,dx=-\int_{\Omega}u_t^{\varepsilon}\div\left(\frac{D u^{\varepsilon}}{\sqrt{\varepsilon^2+|D u^{\varepsilon}|^2}}\right)\,dx,
\end{eqnarray*}
and thus,
\begin{eqnarray*}
\frac{{\rm d}}{{\rm d}t}\int_{\Omega}\sqrt{\varepsilon^2+|D u^{\varepsilon}|^2}\,dx&=&-\int_{\Omega}u_t^{\varepsilon}\left(\frac{u^{\varepsilon}_t}{\sqrt{\varepsilon^2+|Du^{\varepsilon}|^2}}-c(x)\right)\,dx\\
&=&-\int_{\Omega}\left(\frac{(u_t^{\varepsilon})^2}{\sqrt{\varepsilon^2+|D u^{\varepsilon}|^2}}-c(x)u_t^{\varepsilon}\right)\,dx\\
&\leq&-\frac{1}{\sqrt{\varepsilon^2+L^2}}\int_{\Omega}(u_t^{\varepsilon})^2dx+\int_{\Omega}c(x) u_t^{\varepsilon}\,dx.
\end{eqnarray*}

Rearranging the terms, 
$$
\frac{\rm d}{{\rm d}t}\left(\int_{\Omega}\sqrt{\varepsilon^2+|D u^{\varepsilon}|^2}\,dx-\int_{\Omega}c(x) u^{\varepsilon}\,dx\right)\leq-\frac{1}{\sqrt{\varepsilon^2+L^2}}\int_{\Omega}(u_t^{\varepsilon})^2\,dx.
$$
Integrating the inequality above, we have
\begin{eqnarray*}
\int_0^{T}\int_{\Omega}(u_t^{\varepsilon})^2\,dxdt&\leq& \sqrt{\varepsilon^2+L^2}\int_{\Omega}c(x)(u^{\varepsilon}(x,T)-u^{\varepsilon}(x,0))\,dx\\
&+&\sqrt{\varepsilon^2+L^2}\int_{\Omega}\left(\sqrt{\varepsilon^2+|D u^{\varepsilon}|^2(x,0)}-\sqrt{\varepsilon^2+|D u^{\varepsilon}|^2(x,T)}\,\right)\,dx.
\end{eqnarray*}
Note that $\|u\|_{L^\infty(\ol \Omega \times [0,\infty))} \leq \|u_0\|_{L^\infty (\ol \Omega)}$.  
Therefore,
$$
\limsup_{\varepsilon\to0}\int_0^{T}\int_{\Omega}(u_t^{\varepsilon})^2\,dxdt\leq C,
$$
where $C$ is a constant independent of $\varepsilon \in (0,1)$ and $T>0$. 
Hence, we  get that $u^{\varepsilon}_t\rightharpoonup u_t$ weakly in $L^2(\overline{\Omega}\times[0,T])$ as $\varepsilon\rightarrow 0$ for each $T>0$.


By weakly lower semi-continuity, 
\begin{equation*}\label{eq:lowersemi}
\int_0^{T}\int_{\Omega}(u_t)^2\,dxdt\leq \liminf\limits_{\varepsilon\rightarrow 0}\int_0^{T}\int_{\Omega}(u_t^{\varepsilon})^2\,dxdt\leq C.
\end{equation*}
Since the constant $C$ is independent of $\varepsilon,\ T$, we see that
\begin{equation}\label{eq:l2-b}
\int_0^{\infty}\int_{\Omega}(u_t)^2\,dxdt\leq C.
\end{equation}
For every $\{t_k\}\rightarrow\infty$, by the Arzel\`a-Ascoli theorem, there exist a subsequence $\{t_{k_j}\}$ and a Lipschitz continuous function $v$ such that
\[
u_{k_j}(x,t)=u(x,t+t_{k_j})\rightarrow v(x,t),
\]
locally uniformly on $\overline{\Omega}\times[0,\infty)$. In particular,
\begin{equation}\label{eq:converg1}
u_{k_j}(x,t)=u(x,t+t_{k_j})\rightarrow v(x,t),
\end{equation}
uniformly on $\overline{\Omega}\times[0,T]$, for every $T>0$. 
By stability results of viscosity solutions, $v$ satisfies
\begin{equation*}
\begin{cases}
v_t=|D v|\text{div}\left(\frac{D v}{|D v|}\right)+c|D v| \quad &\text{ in }  \Omega\times(0,\infty),\\
\displaystyle \frac{\partial v}{\partial \Vec{\mathbf{n}}}=0 \quad &\text{ on } \partial\Omega\times[0,\infty).
\end{cases}
\end{equation*}
Thanks to \eqref{eq:l2-b}, we have
\[
\int_0^1\int_{\Omega}(u_{k_j})_t^2\,dxdt=\int_{t_{k_j}}^{1+t_{k_j}}\int_{\Omega}(u_t)^2\,dxdt\rightarrow0,
\]
as $j\rightarrow\infty$. This shows that 
\[
(u_{k_j})_t\rightharpoonup0,
\]
weakly in $L^2(\overline{\Omega}\times[0,1])$ as $j\rightarrow\infty$. On the other hand, (\ref{eq:converg1}) implies that
\[
(u_{k_j})_t\rightharpoonup v_t,
\]
weakly in $L^2(\overline{\Omega}\times[0,1])$ as $j\rightarrow\infty$. Consequently, $v_t=0$ weakly, and $v$ is constant in $t$. Thus, $v$ is a solution of \eqref{eq:stationary}, that is, $v$ solves
\begin{equation*}
\begin{cases}
|D v|\div \left(\frac{D v}{|D v|}\right)+c(x)|D v|=0 \quad &\text{ in } \Omega,\\
\displaystyle \frac{\partial v}{\partial \Vec{\mathbf{n}}}=0\quad &\text{ on } \partial\Omega.
\end{cases}
\end{equation*}
Equation \eqref{eq:stationary} has many viscosity solutions in general. 
For example, as $v$ is a solution, $v+C$ is also a solution for any $C\in \R$.
Therefore, $v$ may depend on the choice of subsequence $\{t_k\}_{k}$.

At last, we prove that $v$ is independent of the choice of subsequence $\{t_k\}_{k}$. Since $u_{k_j}$ converges uniformly to $v$ on $\overline{\Omega}\times[0,1]$, for every $\varepsilon>0$ there exists $j$ large enough such that
$$
|u_{k_j}(x,t)-v(x)|<\varepsilon,\quad\ \text{for all} \ (x,t)\in\overline{\Omega}\times[0,1].
$$
In particular, $v(x)-\varepsilon<u_{k_j}(x,0)=u(x,t_{k_j})<v(x)+\varepsilon$ for all $x\in \overline{\Omega}$. 
By the comparison principle, 
$$
v(x)-\varepsilon\leq u(x,t)\leq v(x)+\varepsilon \quad \text{ for $(x,t)\in\overline{\Omega}\times[t_{k_j},\infty)$.}
$$
This implies that $u(\cdot,t)$ converges uniformly to $v$ on $\overline{\Omega}$ without taking a subsequence.
\end{proof}

\section{The large time profile in the radially symmetric setting}\label{sec:blow-up}
In this section, we study the radially symmetric setting and illustrate some examples of multiplicity of solutions to the stationary problem \eqref{eq:stationary}. 
We always assume here \eqref{con:radial}, that is,
\begin{equation*}
\begin{cases}
\Omega = B(0,R) \text{ for some } R>0,\\
c(x) = c(r) \text{ for } |x|=r \in [0, R],\\
u_0(x)= u_0(r)  \text{ for } |x|=r \in [0, R].
\end{cases}
\end{equation*}
Here, $c \in C^1([0,R], [0,\infty))$, and $u_0 \in C^2([0,R])$ with $u_0'(R)=0$ are given.
In this setting, \eqref{eq:stationary} reduces to the following Hamilton-Jacobi equation with Neumann boundary condition
\begin{equation}\label{eq:radialsta}
\begin{cases}
-\frac{n-1}{r}\phi_r-c(r)\lvert\phi_r\rvert=0, \quad &\text{ in } (0,R),\\
\displaystyle\hspace{22mm} \phi_r(R)=0.\quad &\text{} 
\end{cases}
\end{equation}

It is worth noting that no boundary condition is needed at $r=0$, and that the Hamiltonian is concave and maybe noncoercive.
Clearly, every constant is a solution to \eqref{eq:radialsta}.
Also, if $\phi$ is a solution to \eqref{eq:radialsta}, then so is $C\phi$ for any given constant $C\geq 0$.

We have the following proposition.

\begin{proposition}\label{prop:radconst}
Let $\mathcal{A}=\left\{r\in(0,R]\,:\,c(r)= \frac{n-1}{r}\right\}$.
Denote by 
\[
r_{\min} = 
\begin{cases}
\min\{r\,:\, r\in \cA\} >0 \quad &\text{ if }  \cA \neq \emptyset,\\
R \quad &\text{ if }  \cA=\emptyset.
\end{cases}
\]
 Let $\phi$ be a Lipschitz solution to \eqref{eq:radialsta}. 
 Then, $\phi$ is constant on each connected component of $(0,R)\setminus\mathrm{int}(\mathcal{A})$.
 In particular, $\phi$ is constant on $[0,r_{\min}]$.
\end{proposition}

\begin{proof}
Factoring \eqref{eq:radialsta} into $\left(-\frac{n-1}{r}\pm c(r)\right)\phi_r(r)=0$, we see that either $-\frac{n-1}{r}\pm c(r)=0$ or $\phi_r(r)=0$ at each point of differentiability of $\phi$. 

Take $(a,b) \subset \left((0,R)\setminus\mathrm{int}(\mathcal{A})\right)$ for some $a<b$.
By the above, we have that $\phi_r(r)=0$ for a.e. $r\in (a,b)$, and thus, $\phi$ is constant on $[a,b]$.
\end{proof}

\begin{example}[A toy model]\label{ex}
We consider the case that $c(r)$ is of the form
$$
c(r)=\left\{\begin{array}{ll}
\frac{n-1}{a},\quad\quad\quad\hspace{1mm} &0\leq r< a,\\
\frac{n-1}{r},\quad\quad\quad\hspace{1.5mm} &a\leq r\leq b,\\
\frac{n-1}{b},\quad\quad\quad\hspace{1.5mm} &b< r\leq R,
\end{array}\right.
$$
for some $0<a<b<R,$ then the stationary problem \eqref{eq:radialsta} admits multiple solutions of the form
$$
\phi(r)=\left\{\begin{array}{ll}
c_1,\quad\quad\quad\hspace{4mm} &0\leq r\leq a,\\
g(r),\quad\quad\quad &a\leq r\leq b,\\
c_2,\quad\quad\quad\hspace{4mm} &b\leq r\leq R,
\end{array}\right.
$$
where $c_1\geq c_2$ are constants, $g(r)$ is any nonincreasing function on $[a,b]$ with $g(a)=c_1,\ g(b)=c_2.$ Here, the function $g$ can be discontinuous if we extend the definition of viscosity solutions to discontinuous functions (see \cite{G} for instance).
\end{example}

Example \ref{ex} shows further the multiplicity of solutions to \eqref{eq:radialsta} besides the constant functions noted above. 
Thus, it is important to address how the large-time limit $\phi_{\infty}$ depends on the initial data $u_0$. In this radially symmetric setting, we are able to characterize the limiting profile and specify its dependence on the initial data.

Equations \eqref{eq:levelset1}--\eqref{eq:levelset3} become
\begin{equation*}\label{eq:radial}
\begin{cases}
\phi_t-\frac{n-1}{r}\phi_r-c(r)\lvert \phi_r\rvert=0 \quad &\text{ in } (0,R)\times(0,\infty),\\
\displaystyle \hspace{27.5mm} \phi_r(R,t)=0\quad &\text{ for } t \geq 0, \\
\hspace{26.75mm}\phi(r,0)=u_0(r) \quad &\text{ for } r\in[0,R].
\end{cases}
\end{equation*}
Here, $u(x,t)=\phi(|x|,t)$ for $(x,t)\in B(0,R) \times [0,\infty)$.
Note that this is a first-order Hamilton-Jacobi equation with a concave Hamiltonian.
The associated Lagrangian $L=L(r,q)$ to the Hamiltonian $H(r,p)=-\frac{n-1}{r}p-c(r)\lvert p\rvert$ is
\begin{align*}
L(r,q)&=\inf_{p\in\mathbb{R}}\left\{p\cdot q-\left(-\frac{n-1}{r}p-c(r)\lvert p\rvert\right)\right\}\\
&=\inf_{p\in\mathbb{R}}\left\{\left(q+\frac{n-1}{r}\right)p+c(r)\lvert p\rvert\right\}\\
&=\left\{\begin{array}{ll}
0,\quad\quad\quad\hspace{1.5mm} \textrm{if}\ \left\lvert q+\frac{n-1}{r}\right\rvert\leq c(r),\\
-\infty,\quad\quad\textrm{otherwise}.
\end{array}\right.
\end{align*}
Therefore, we have the following representation formula for $\phi=\phi(r,t)$
$$
\phi(r,t)=\sup\left\{u_0(\gamma(0))\,:\,(\gamma,v,l)\in\mathrm{SP}(r,t)\right\},
$$
where we denote by $\mathrm{SP}(r,t)$ the Skorokhod problem. 
For a given $r\in(0,R],\ v\in L^{\infty}([0,t])$, the Skorokhod problem seeks to find a solution $(\gamma,l)\in \mathrm{Lip}((0,t))\times L^{\infty}((0,t))$ such that
$$
\left\{\begin{array}{ll}
\gamma(t)=r,\quad\quad\hspace{1.5mm} \gamma([0,t])\subset(0,R],\\
l(s)\geq0\quad\quad\quad \textrm{for almost every}\ s>0,\\
l(s)=0\quad\quad\quad \textrm{if}\ \gamma(s)\neq R,\\
\left|-v(s)+\frac{n-1}{\gamma(s)}\right|\leq c(\gamma(s)),\\
v(s)=-\dot{\gamma}(s)+l(s)n(\gamma(s)),
\end{array}\right.
$$
and the set $\mathrm{SP}(r,t)$ collects all the associated triples $(\gamma,v,l).$
Here, $n(R)=1$ is the outward normal vector  to $(0,R)$ at $R$. 
See \cite[Theorem 4.2]{I2} for the existence of solutions of the Skorokhod problem and \cite[Theorem 5.1]{I2} for the representation formula.
See \cite{GMT} for a related problem on large time behavior and large time profile.

\begin{example}\label{ex2}
Consider Example \ref{ex}. 
To recall, $c(r)$ is defined in the following way
$$
c(r)=\left\{\begin{array}{ll}
\frac{n-1}{a},\quad\quad\quad\hspace{1mm} &0\leq r < a,\\
\frac{n-1}{r},\quad\quad\quad\hspace{1.5mm} &a\leq r\leq b,\\
\frac{n-1}{b},\quad\quad\quad\hspace{1.5mm} &b < r\leq R.
\end{array}\right.
$$
for some $0<a<b<R.$ We analyze the velocity condition $\left\lvert\dot{\gamma}(s)+\frac{n-1}{\gamma(s)}\right\rvert\leq c(\gamma(s))$. Note that $c(r)$ is less than $\frac{n-1}{r}$, equal to $\frac{n-1}{r}$, and greater than $\frac{n-1}{r}$ in the written order, respectively. In each case, then, the velocity condition becomes
\begin{equation*}
\begin{cases}
-\frac{n-1}{a}-\frac{n-1}{\gamma(s)}\leq\dot{\gamma}(s)\leq\frac{n-1}{a}-\frac{n-1}{\gamma(s)}<0, \quad &0<\gamma(s)<a,\\
\hspace{8mm}-\frac{2(n-1)}{\gamma(s)}\leq\dot{\gamma}(s)\leq0,\quad &a\leq\gamma(s)\leq b,\\
-\frac{n-1}{b}-\frac{n-1}{\gamma(s)}\leq\dot{\gamma}(s)\leq\frac{n-1}{b}-\frac{n-1}{\gamma(s)}, \quad &b\leq\gamma(s)<R.
\end{cases}
\end{equation*}

\begin{figure}[htbp]
	\begin{center}
            \includegraphics[height=7cm]{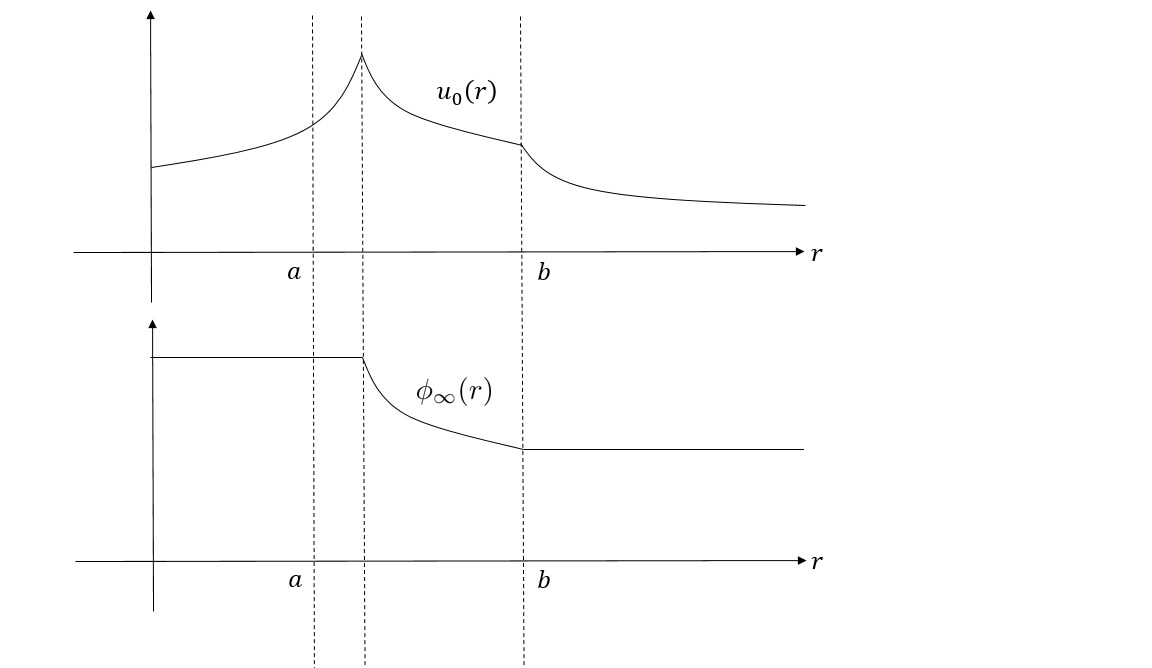}
		\vskip 0pt
		\caption{Stationary solution of (\ref{eq:radialsta})}
        \label{fig:radialsta}
	\end{center}
\end{figure}
Focusing the right hand side in each case, we see that the point $\gamma(s)$ must move left as time $s$ increases, can stay still, and can go right in the written order, respectively. 
This point of view in terms of the Lagrangian $L(r,q)$ and Proposition \ref{prop:radconst} explain the limit $\phi_{\infty}(r)$ of $\phi(r,t)$ as $t\to\infty$ in the above illustration of Figure \ref{fig:radialsta}.

\end{example}

The description in Example \ref{ex2} shows how to formulate and write the limit $\phi_{\infty}$ in terms of the initial data $u_0$ in full generality.
 We note one more thing on the boundary.
If $c(h)<\frac{n-1}{h}$ for all $h \in (0,R]$, then the reversed curve $\eta(s):=\gamma(t-s)$ of an admissible curve $\gamma$ must go right, and it stays on the boundary $r=R$ once it reaches there. 
 This is where the effect of the Skorokhod problem comes in, and it means that the solution $\phi(r,t)$ needs to be understood in the sense of viscosity solutions. 
 We also note that in this setting, we can prove that $\phi$ is same as the value function of the state constraint problem. 
 Together with this observation on the boundary, analyzing curves $\gamma(s)$ explains how the limit $\phi_{\infty}$ depends on the initial data $u_0$, and indeed the analysis of admissible curves yields the proof of Theorem \ref{thm:radlimit}.

\smallskip

We now give some preparation steps in order to prove Theorem \ref{thm:radlimit}. 
Let $\eta(s):=\gamma(t-s),\ s\in[0,t],$ be the reversed curve of a curve $\gamma\in\AC([0,t],(0,R])$ with $(\gamma,v,l)\in\textrm{SP}(r,t)$. 
Then, we have the following velocity condition for $\eta$
\begin{equation}\label{velocityeta}
-c(\eta(s))+\frac{n-1}{\eta(s)}\leq \dot{\eta}(s)\leq c(\eta(s))+\frac{n-1}{\eta(s)}\quad\textrm{for a.e. }s\in[0,t]\ \textrm{with }\eta(s)\neq R.
\end{equation}

The following lemma is a direct consequence of the comparison principle.
\begin{lemma}\label{lemma:5-1}
Let $r_0\in(0,R)$. 
Let $\eta_1\in\AC([0,\infty),(0,R])$ be a curve satisfying
\begin{equation*}
\begin{cases}
\dot{\eta}_1(s)=-c(\eta_1(s))+\frac{n-1}{\eta_1(s)}, \quad &\textrm{for $s>0$ provided that $\eta_1(s)<R$},\\
\eta_1(0)=r_0.&
\end{cases}
\end{equation*}
If $\eta_1(s_0)=R$ for some $s_0>0$, then we set $\eta_1(s)=R$ for all $s\geq s_0$.

For each $t>0$, let $\eta\in\AC([0,t],(0,R])$ be the reversed curve given above with $\eta(0)\geq r_0$. 
Then, $\eta_1(s)\leq \eta(s)$ for all $s\in[0,t]$.
\end{lemma}
%

\begin{lemma}\label{lemma:5-2}
Assume the settings of Theorem \ref{thm:radlimit} and Lemma \ref{lemma:5-1}.
Then, 
\begin{equation}\label{eq:lim-eta-1}
\lim_{s\to \infty} \eta_1(s) = d(r_0).
\end{equation}
\end{lemma}
\begin{proof}
If $r_0 \in \cA$, then $\eta_1(s)= r_0$ for all $s\geq0$, and hence \eqref{eq:lim-eta-1} holds.
 
 Next, we only need to consider the case that $r_0 \in \cA_+$ as the proof of the case that $r_0 \in \cA_-$ follows analogously.
It is clear that $\eta_1$ is decreasing, and by Lemma \ref{lemma:5-1}, $\eta_1(s)\geq d(r_0)$ for all $s\geq0$. 
Therefore, $\lim_{s\to\infty}\eta_1(s)$ exists, and
\[
\lim_{s\to\infty}\eta_1(s)=r_1 \geq d(r_0).
\]
This yields further that
\[
\limsup_{s\to\infty}\dot{\eta}_1(s)=0.
\] 
Hence,
\[-c(r_1)+\frac{n-1}{r_1}=0,
\]
which implies that $r_1=d(r_0)$.

\end{proof}

\begin{proof}[Proof of Theorem \ref{thm:radlimit}.]
For $(r_0,t)\in (0,R) \times [0,\infty)$, we have
\[
\phi(r_0,t)=\sup\{u_0(\eta(t)):(\gamma,v,l)\in\textrm{SP}(r_0,t),\ \eta(s)=\gamma(t-s),\ s\in[0,t]\}.
\]
We say that $\eta \in\AC([0,t],(0,R])$ is admissible if $\eta(s)=\gamma(t-s),\ s\in[0,t]$ for some $\gamma$ with $(\gamma,v,l)\in\textrm{SP}(r_0,t)$. 
Let $\eta_1$ be the curve given in the statement of Lemma \ref{lemma:5-1}.
By Lemma \ref{lemma:5-1}, $\eta(s)\geq\eta_1(s)$ for $s\in[0,t]$ for any admissible curve $\eta$.  
From this fact, we see that
\[
\phi(r_0,t)\leq\sup\{u_0(r):r\geq\eta_1(t)\},
\]
and therefore, by Lemma \ref{lemma:5-2},
\[
\limsup_{t\to\infty}\phi(r_0,t)\leq\max\{u_0(r):r\geq d(r_0)\}.
\]

In order to complete the proof, it suffices to show the other direction
\begin{equation}\label{eq:fin}
\liminf_{t\to\infty}\phi(r_0,t)\geq\max\{u_0(r):r\geq d(r_0)\}.
\end{equation}
To show this, let $r_1\in[d(r_0),R]$ be such that 
\[
u_0(r_1)=\max\{u_0(r):r\geq d(r_0)\}.
\]

We consider first the case $r_0 \in \cA$.
Then, $r_1\geq r_0$.
Let $\eta_2$ solve
\begin{equation*}
\begin{cases}
\dot{\eta}_2(s)=c(\eta_2(s))+\frac{n-1}{\eta_2(s)}, \quad &\textrm{for }s>0,\\
\eta_2(0)=r_0.&
\end{cases}
\end{equation*}
Note that $c(r) + (n-1)/r \geq (n-1)/R>0$ for all $r\in (0,R]$.
Then, there is a unique number $t_2\geq0$ such that $\eta_2(t_2)=r_1$. 
Now, for $t\geq t_2$,  let $\eta$ be defined as
\begin{equation*}
\eta(s)=
\begin{cases}
r_0, \quad &\textrm{if }s\leq t-t_2,\\
\eta_2(s-(t-t_2)), \quad&\textrm{if }s\geq t-t_2.
\end{cases}
\end{equation*}
Then, $\eta$ is admissible, and $\phi(r_0,t) \geq u_0(\eta(t))=u_0(r_1)$. 
Thus, \eqref{eq:fin} holds.

Next, we consider the case $r_0\in \cA_+$.
If  $r_1\geq r_0$, then we repeat the above process to conclude.
If $r_1<r_0$, then $r_1\in[d(r_0),r_0)$ necessarily, and in this case, we use the curve $\eta_1$. 
We note that if $r_1>d(r_0)$, then there is a unique number $t_1\geq0$ such that $\eta_1(t_1)=r_1$. 
Now, for $t\geq t_1$,   let $\eta$ be defined as
\begin{equation*}
\eta(s)=
\begin{cases}
r_0, \quad &\textrm{if }s\leq t-t_1,\\
\eta_1(s-(t-t_1)), \quad&\textrm{if }s\geq t-t_1.
\end{cases}
\end{equation*}
Then, the curve $\eta$ is admissible, and $\phi(r_0,t) \geq u_0(\eta(t))=u_0(r_1)$. 
If $r_1=d(r_0)$, we take $\eta=\eta_1$ and recall that $\lim_{t\to\infty}\eta_1(t)=d(r_0)$, which gives  $\phi(r_0,t) \geq u_0(\eta(t)) \to u_0(r_1)$ as $t\to\infty$.
Therefore, \eqref{eq:fin} holds.

Finally, we study the case $r_0\in \cA_-$.
Let $\eta_2, t_2$ be defined as above.
There exists a unique $t_3>0$ such that $\eta_2(t_3)=d(r_0)$.
In this case, $r_1 \geq d(r_0)$ and $t_2\geq t_3$.
For $t\geq t_2$, define
\begin{equation*}
\eta(s)=
\begin{cases}
\eta_2(s), \quad &\textrm{if } 0\leq s \leq t_3,\\
d(r_0), \quad &\textrm{if } t_3\leq s \leq t-(t_2-t_3),\\
\eta_2(s-(t-t_2)), \quad&\textrm{if } t-(t_2-t_3) \leq s\leq t.
\end{cases}
\end{equation*}
Then, $\eta$ is admissible, and $\eta(t)=r_1$, which yields \eqref{eq:fin}.
\end{proof}

Next, we prove Corollary \ref{thm:non-Lip-rad}, and discuss the sharpness of condition \eqref{condition:c}. 

\begin{proof}[Proof of Corollary \ref{thm:non-Lip-rad}]
The values of $\phi_{\infty}$ are computed directly from Theorem \ref{thm:radlimit}. 
This tells us the fact that the solution $u=u(r,t)$ is not globally Lipschitz because if it were globally Lipschitz, then  the limit $\phi_{\infty}$ would be as well.
\end{proof}

Corollary \ref{thm:non-Lip-rad} realizes a jump discontinuity in the limit, which indicates that condition \eqref{condition:c}, which is needed for the globally Lipschitz continuity of $u$, is almost optimal.
As the domain $\Omega=B(0,R)$ is convex, $C_0 \leq 0$, and \eqref{condition:c} becomes $\frac{1}{n} c(x)^2 - |Dc(x)| -\delta>0$.
Let us now assume that $c(r)$ touches $\frac{n-1}{r}$ from below at $a$.
Then, 
\[
c(a)  = \frac{n-1}{a} \quad \text{ and } \quad c'(a) = -\frac{n-1}{a^2}.
\] 
At $r=a$, we see that
\[
\frac{1}{n} c(a)^2  - |c'(a)| = \frac{(n-1)^2}{n a^2} - \frac{n-1}{a^2}=-\frac{n-1}{n a^2} <0.
\]
Moreover, we see that condition \eqref{condition:c} is essentially optimal if we seek to find sufficient conditions on the force $c$ that are uniform in dimensions $n$ and in $R$ because the left hand side of the above goes to zero as $a \to \infty$.


\section{The gradient growth as time tends to infinity in two dimensions} \label{sec:blow-up2}
Let $n=2$.
Let the forcing term $c$ be a positive constant in $\Omega$, that is, $c(x)=c$ for all $x\in \ol \Omega$ for some $c>0$.
Consider the following nonconvex domain,
\begin{equation}\label{set:omega}
\Omega = \{ (x_1, x_2) \in \mathbb{R}^2 : |x_2| < f(x_1)\},
\end{equation}
where $f(x) = \frac{m}{2} x^2 + k$ for fixed $m>0$ and $k>0$. 
Here, $\Omega$ is unbounded.

In this unbounded setting, let $R_0>0$  be a sufficiently large constant.
Let $\widetilde \Omega \subset \R^n$ be a bounded $C^{2,\theta}$ domain such that
\[
\Omega \cap B(0,R_0) \subset \widetilde \Omega \subset \Omega.
\]
We say that $u$ is a solution (resp., subsolution, supersolution) of \eqref{eq:levelset1}--\eqref{eq:levelset3} on $\ol \Omega \times [0,\infty)$ if there exists $\alpha\in\mathbb{R}$ such that 
\begin{equation}\label{condition:cp}
u - \alpha  =u_0-\alpha= 0 \quad \text{ on } (\ol \Omega \setminus B(0,R_0)) \times [0,\infty),
\end{equation}
and $u$ is a solution (resp., subsolution, supersolution) of \eqref{eq:levelset1}--\eqref{eq:levelset3} with $\widetilde \Omega$ in place of $\Omega$.

Let $u$ be the solution to  \eqref{eq:levelset1}--\eqref{eq:levelset3}.
If  a level set of $u$ is a smooth curve, then it is evolved by the forced curvature flow equation $V=\kappa +c$, where $V$ is the normal velocity and $\kappa$ is the curvature in the direction of the normal.
Then, the classical Neumann boundary condition becomes the right angle condition for the level-set curves with respect to $\partial \Omega$, that is, if a smooth level curve and $\partial \Omega$ intersect, then their normal vectors are perpendicular at the points of intersections.

We show that if $c$ is too small and fails to satisfy  \eqref{condition:c}, then there exist discontinuous viscosity solutions to \eqref{eq:stationary}. 
In particular, we find that one such discontinuous solution of \eqref{eq:stationary}  is stable in the sense that the solution of \eqref{eq:levelset1}--\eqref{eq:levelset3}  with a suitable choice of initial data converges to this discontinuous stationary solution as time goes to infinity. 
This implies that the global Lipschitz estimate for the solution of  \eqref{eq:levelset1}--\eqref{eq:levelset3} does not hold.
The following is the main result of this section.

\begin{theorem}
\label{thm:conv}
Let $\Omega$ be the set given by \eqref{set:omega}, and $c(x)=c$ for all $x\in \ol \Omega$ for 
$c \in (0,r_{\min}^{-1})$, 
where $r_{\min}$ is defined by \eqref{func:rmin}. 
Let $u \in C(\overline{\Omega} \times [0, \infty))$ be the solution of \eqref{eq:levelset1}--\eqref{eq:levelset3} with the given initial data $u_0\in C^{2,\theta}(\overline{\Omega})$ satisfying that $\frac{\partial u_0}{\partial \Vec{\mathbf{n}}}=0\ \text{on}\ \partial\Omega$ and there exist constants $l_1, l_2, \alpha$ and $\beta$ such that $l_1 \in (0, a_1)$, $l_2 \in (0, a_2 - a_1)$, $\alpha < \beta$,
\begin{align}
\label{eqn:1conv}
u_0(x)=
\begin{cases}
\beta \quad &\text{ for } x=(x_1,x_2) \in U(a_1 - l_1),\\
\alpha \quad &\text{ for } x=(x_1,x_2) \in \overline{\Omega}\setminus\overline{U(a_1+l_2)},
\end{cases}
\end{align}
and $ \alpha \leq u_0 \leq \beta$, 
where $U(a)$ is defined by \eqref{eqn:ua} for $a>0$, and $0<a_1<a_2$ is given in Theorem \ref{thm:sta}. 
Then,
\[
\lim_{t \rightarrow \infty} u(x,t) = 
\begin{cases}
\beta \quad &\hbox{ if }x \in U(a_1),\\
\alpha \quad &\hbox{ if }x \in \overline{\Omega}\setminus\overline{U(a_1)}.
\end{cases}
\]
\end{theorem}

\subsection{Set-theoretic stationary solutions}

For $a>0$, consider a family of curves with constant curvature in $\Omega$,
\begin{align}
\label{eqn:cur}
X(a, \theta) = (X_1(a, \theta), X_2(a, \theta)) = p(a) + r(a) (\cos \theta, \sin \theta), \ \ \ |\theta| < \arctan(ma),
\end{align}
where we choose $p(a)$, $r(a)$ so that the curve 
\[
\Gamma:=\{(X_1(a, \theta), X_2(a, \theta))\,:\, |\theta| < \arctan(ma)\}\cup
\{(-X_1(a, \theta), X_2(a, \theta))\,:\, |\theta| < \arctan(ma)\}
\] 
has a constant curvature, and is perpendicular to the boundary $\partial\Omega$. 
Indeed, set  
\[
p(a):=\left(\frac{a}{2}-\frac{k}{ma},0\right). 
\]
Then, we see that the tangent line for $\{(x_1,x_2)\mid x_2=f(x_1)\}$ at $x_1=a$ goes through $p(a)$.  
Moreover, setting 
\[
r(a):=\left|\left(a,\frac{ma^2}{2}+k\right)-p(a)\right|=\left(\frac{a}{2}+\frac{k}{ma}\right)\sqrt{m^2a^2+1}, 
\]
by elementary geometry, we can check that 
\begin{equation*}\label{eq:Neumann}
\Gamma\bot \partial\Omega. 
\end{equation*} 
The parameter $a$ will be specified so that 
\[
c=\frac{1}{r(a)}
\]
in Lemma \ref{lem:ss}.

\begin{figure}[h]\label{illust}
	\centering
	\begin{subfigure}[t]{0.45\textwidth}
	\includegraphics[width=\textwidth]{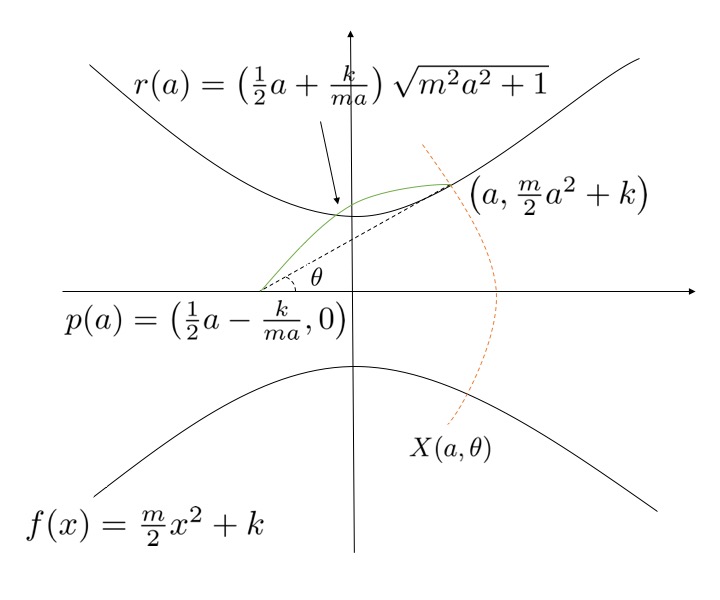}
	\end{subfigure}
	\hspace{0.5cm}
	\begin{subfigure}[t]{0.45\textwidth}
	\includegraphics[width=\textwidth]{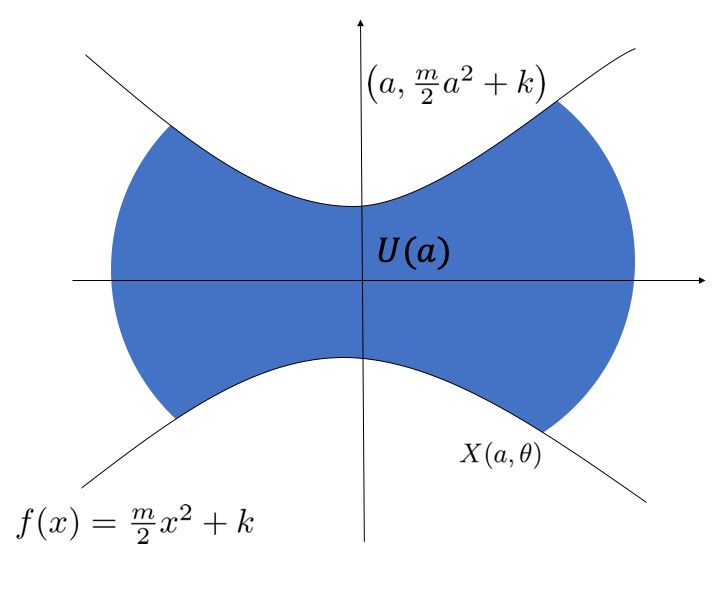}
	\end{subfigure}
	\caption{Illustrations of \eqref{eqn:cur} and \eqref{eqn:ua}}
\end{figure}

The following definition is taken from \cite[Definition 5.1.1]{G}.
\begin{definition}
\label{def:set}
Let $G$ be a set in $\mathbb{R}^n \times J$, where $J$ is an open interval in $(0,T)$. 
We say that $G$ is a set-theoretic subsolution (resp., supersolution) of 
\begin{equation}\label{eq:sur}
V=\kappa+c \quad\text{on} \ \Gamma_t\quad\text{with} \quad \Gamma_t\bot \partial\Omega 
\end{equation}
if $\chi_G^{\ast}$ is a viscosity subsolution (resp., $(\chi_G)_\ast$ is a viscosity supersolution) of \eqref{eq:levelset1}--\eqref{eq:levelset2} in $\mathbb{R}^n \times J$, 
where $\chi_G(x,t)=1$ if $(x,t)\in G$, and $\chi_G(x,t)=0$ if $(x,t)\not\in G$, and $\chi_G^{\ast}$ and $(\chi_G)_\ast$ denote the upper semicontinuous envelope and the lower semicontinuous envelope of $\chi_G$, respectively.  
If $G$ is both a set-theoretic subsolution and supersolution of \eqref{eq:sur}, $G$ is called a set-theoretic solution of \eqref{eq:sur}.
\end{definition}

Set
\begin{align}\label{eqn:ua}
U(a) := \{ (x_1,x_2) \in \Omega : |x_1| < X_1(a, \theta), |x_2| < X_2(a, \theta), |\theta| < \arctan(ma) \}, 
\end{align}
and
\begin{equation}\label{func:rmin}
r_{\min}:= \inf\{r(a):a>0\}.
\end{equation} 
Then, $r_{\min}$ is positive since $r$ is a continuous positive function in $(0, \infty)$ and 
\begin{align}
\label{eqn:rmin}
\lim_{a \rightarrow 0} r(a) = \lim_{a \rightarrow \infty} r(a) = \infty. 
\end{align}
Moreover, by direct computation, we have
\[
r'(a) = \frac{1}{\sqrt{m^2 a^2 + 1}} \left( m^2 a^2 + \frac{1}{2} - \frac{k}{ma^2} \right).
\]
Therefore, $r$ has only one critical point $a_* = \frac{1}{2m} \sqrt{ -1 + \sqrt{1+ 16mk}}$ in $(0, \infty)$ and $r_{\min} = r(a_*)$. 
In addition,
\begin{align}
\label{eqn:sta2}
\hbox{$r'(a)<0$ if $a<a_*$, and $r'(a)>0$ if $a>a_*$.}
\end{align}

\begin{lemma}
\label{lem:ss}
If $c = \frac{1}{r(a)}$ for some $a>0$, then $U(a)$ is a set-theoretic stationary solution of \eqref{eq:levelset1}--\eqref{eq:levelset2}.
\end{lemma}

\begin{proof}
As a consequence of the nice characterization of set-theoretic solutions in \cite[Theorem 5.1.2]{G}, $U(a)$ is a set-theoretic stationary solution of \eqref{eq:sur}  if and only if $0 = \kappa + c$ on $\partial U(a) \cap \Omega$ and the right angle condition holds. 
The equality follows from the fact that $\partial U(a) \cap \Omega$ contains two arcs of two circles of the same radius $r(a)$ and curvature $\kappa = -r(a)^{-1} = -c$. 

On the other hand, these arcs intersect with $\partial \Omega$ at four points $( a, \pm f(a))$, $(-a, \pm f(a))$. 
By symmetry, it suffices to prove the right angle condition at $(a,f(a))$. Notice that 
\begin{align*}
(a,f(a)) =  (X_1(a, \arctan(ma)), X_2(a, \arctan(ma))) = p(a) + \frac{r(a)}{\sqrt{m^2 a^2 + 1}} \cdot (1, ma).
\end{align*} 
Therefore, the line joining $(a,f(a))$ and $p(a)$, the center of the arc, is tangent to $\partial \Omega$ at $(a,f(a))$. 
Thus, $\partial U(a) \cap \Omega$ satisfies the right angle condition at $(a,f(a))$.
\end{proof}

\begin{theorem}\label{thm:sta}
If $c \in (0,\frac{1}{r_{\min}})$, then there exist two positive constants $a_1 < a_2$ such that $U(a_i)$ is a set-theoretic stationary solution of \eqref{eq:sur} for $i=1,2$.
\end{theorem}

\begin{proof}
Thanks to \eqref{func:rmin}--\eqref{eqn:sta2}, there exist two positive constants $a_1, a_2$ with $a_1<a_*<a_2$ such that
\begin{align}
\label{eqn:sta}
r(a_1) = r(a_2) = \frac{1}{c}.
\end{align}
By Lemma~\ref{lem:ss}, $U(a_i)$ is a set-theoretic stationary solution of \eqref{eq:sur} for $i=1,2$.
\end{proof}

\subsection{Stability}
Let $a_i$ be the constants given by Theorem \ref{thm:sta} for $i=1,2$. 
In this section, we prove that $U(a_1)$ given by \eqref{eqn:ua} is a set-theoretic solution which is stable in the sense of Theorem \ref{thm:conv}.

\begin{lemma}\label{lem:sup}
Let $l_1\in(0,a_1)$, $l_2\in(0,a_2-a_1)$ and $\delta>0$. 
Set $\underline{a}(t):= a_1 - l_1 e^{-\delta t}$ 
and $\overline{a}(t):= a_1+l_2 e^{-\delta t}$.  
There exists $\delta_0=\delta_0(m,k,l_1,l_2)$ such that 
$U(\underline{a}(t))$ and $U(\overline{a}(t))$ are 
a set-theoretic subsolution and supersolution to \eqref{eq:sur} for all $\delta\in(0,\delta_0)$, respectively. 
\end{lemma}

\begin{proof}
We only prove that $U(\underline{a}(t))$ is a set-theoretic subsolution, since we can similarly prove that $U(\overline{a}(t))$ is a set-theoretic supersolution. 
Let $\tilde{X}(t):=X(\underline{a}(t),\theta)$. 
From the characterization of set-theoretic solutions in \cite[Theorem 5.1.2]{G}, it suffices to show that for $t \geq 0$,
\begin{align}\label{eqn:sup1}
\frac{d \tilde{X}}{d t} \cdot \Vec{\mathbf{n}} \leq -\frac{1}{r(\underline{a}(t))} + c \quad  \hbox{ for all } t>0, 
\end{align}
where 
$\Vec{\mathbf{n}}$ is the outward normal vector $\Vec{\mathbf{n}}$ of $U(\underline{a}(t))$, that is, 
$\Vec{\mathbf{n}}=(\cos \theta, \sin \theta)$.

Note that 
\[
\frac{d\tilde{X}}{dt} \cdot \Vec{\mathbf{n}} = \frac{\partial \underline{a}}{\partial t} \frac{\partial {X}}{\partial a} \cdot \Vec{\mathbf{n}} 
=
\delta l_1e^{-\delta t}  \frac{\partial {X}}{\partial a} \cdot \Vec{\mathbf{n}}
=\delta (a_1 - \underline{a}(t)) \frac{\partial {X}}{\partial a} \cdot\Vec{\mathbf{n}}.  
\]
Also, for any constant $L>0$, there exists $C = C(m,k,L)>0$ such that
\begin{align*}
&\frac{\partial {X}}{\partial a}(a,\theta) \cdot\Vec{\mathbf{n}}= p'(a) \cdot \Vec{\mathbf{n}} + r'(a) 
\le 
 |p'(a)| + r'(a) \\
 =&\, 
 \frac{1}{2}  + \frac{m^2 a^2 + \frac{1}{2}}{\sqrt{m^2 a^2 +1}} + \frac{mk}{m^2 a^2 +1 + \sqrt{m^2 a^2 +1}}
 \le C  
\end{align*}
for all $a\in(0,L)$ and $\theta\in(-\frac{\pi}{2}, \frac{\pi}{2})$.  
Therefore, 
\begin{align*}
\frac{d \tilde{X}}{d t} \cdot \Vec{\mathbf{n}} 
= \delta (a_1 - \underline{a}(t)) \frac{\partial {X}}{\partial a} \cdot\Vec{\mathbf{n}} \leq C\delta (a_1 - \underline{a}(t)).
\end{align*}
The observation \eqref{eqn:sta2} implies that $r(\underline{a}(t)) >r(a_1) = c^{-1}$ for all $t\geq0$, and thus we get 
\begin{align*}
\left( \frac{d \tilde{X}}{d t} \cdot \Vec{\mathbf{n}} \right) \left(-\frac{1}{r(\underline{a}(t))} + c \right)^{-1} \leq \delta C \frac{a_1-\underline{a}(t)}{\frac{1}{r(a_1)}-\frac{1}{r(\underline{a}(t))}}. 
\end{align*}
Thus, \eqref{eqn:sup1} holds for $\delta \in (0, \delta_0)$, where 
\begin{align*}
\label{eqn:sup5}
\delta_0 := \left( C \sup_{a \in [a_1 - l_1, a_1 + l_2]} h(a) \right)^{-1}.
\end{align*}
Here the function $h:[a_1 - l_1, a_1 + l_2] \rightarrow \mathbb{R}$ is given by
\[
h(a) := 
\begin{cases}
\dfrac{a_1-a}{\frac{1}{r(a_1)}-\frac{1}{r(a)}} \quad &\hbox{ for } a \in [a_1 - l_1, a_1 + l_2] \setminus \{a_1\},\\ \\
\dfrac{-r^2(a_1)}{r'(a_1)} \quad &\hbox{ for } a = a_1.
\end{cases}
\]
Since $a_1 + l_2 < a_2$, by \eqref{eqn:sta2} we have $r(a) \neq r(a_1)$ in $[a_1 - l_1, a_1 + l_2] \setminus \{a_1\}$ and $r'(a_1)<0$. 
Therefore, $h$ is well-defined and continuous in $[a_1 - l_1, a_1 + l_2]$. 
Thus, $h$ is bounded in $[a_1 - l_1, a_1 + l_2]$, and hence, $\delta_0>0$ is well-defined, 
which implies that \eqref{eqn:sup1} holds for all $\delta\in (0,\delta_0)$.
\end{proof}

\begin{proof}[Proof of Theorem~\ref{thm:conv}]
We let $\alpha = 0$ and $\beta = 1$ for simplicity. 
Set 
\begin{align*}
\ul{u}(x,t):= \chi_{\overline{U(\underline{a}(t))}}(x) \quad \hbox{ and } \quad \ol{u}(x,t):= \chi_{U(\overline{a}(t))}(x)
\end{align*}
for $(x,t) \in \ol \Omega \times [0,\infty)$, where $\underline{a}$ and $\overline{a}$ are the functions defined in Lemma \ref{lem:sup}.  
By Lemma \ref{lem:sup}, 
we see that   
$\ul{u}$ and $\ol{u}$ are a subsolution and a supersolution of \eqref{eq:levelset1}--\eqref{eq:levelset2}, respectively. 
Due to \eqref{eqn:1conv}, we get
\begin{align*}
\ul{u}(\cdot, 0) = \chi_{\overline{U(\underline{a}(0))}} \leq u_0 \leq \chi_{U(\overline{a}(0))} = \ol{u}(\cdot, 0) \hbox{ on } \overline\Omega.
\end{align*}

In addition, since $$U(a) \subset V(a):= [-(|p(a)|+r(a)),|p(a)|+r(a)] \times [-f(a), f(a)]$$ by construction for $p(a)$ and $r(a)$ given in \eqref{eqn:cur} and $f(a) = \frac{m}{2} a^2 + k$, we obtain
\begin{align*}
\supp (\ul{u}) \subset \bigcup_{a \in [a_1 - l_1, a_1]} V(a) \times [0,\infty)  \hbox{ and } \supp (\ol{u}) \subset \bigcup_{a \in [a_1, a_1 + l_2]} V(a) \times [0,\infty).
\end{align*}
As $|p(\cdot)|+r(\cdot)$ and $f$ are continuous on $[a_1 - l_1, a_1 + l_2]$, there exists a constant $R_0>0$ satisfying \eqref{condition:cp}.

By the comparison principle for \eqref{eq:levelset1}--\eqref{eq:levelset3}, Proposition~\ref{prop:comp}, we get 
\begin{align*}
\label{eqn:conv1}
\ul{u}(\cdot, t) \leq u(\cdot, t) \leq \ol{u}(\cdot, t) 
\quad\text{on} \ \ol \Omega \quad \text{for all} \ t>0. 
\end{align*}
On the other hand, since both $a_1 - l_1 e^{-\delta t}$ and $a_1 + l_2 e^{-\delta t}$ converge to $a_1$ as $t$ goes to infinity,
\begin{align*}
\lim_{t \rightarrow \infty} \ul{u}(x,t) = \lim_{t \rightarrow \infty} \ol{u}(x,t) = 1\quad \hbox{ for } x \in U(a_1),
\end{align*}
and 
\begin{align*}
\lim_{t \rightarrow \infty} \ul{u}(x,t) = \lim_{t \rightarrow \infty} \ol{u}(x,t) = 0 \quad \hbox{ for } x \in \overline{\Omega}\setminus\overline{U(a_1)}, 
\end{align*}
which finish the proof. 
\end{proof}

\section*{Acknowledgements} 
The authors are extremely grateful to the anonymous referee for his/her careful reading and very constructive comments, which help much to improve the presentation of the paper.



\end{document}